\documentclass[11pt]{amsart}
\usepackage{graphicx}
\usepackage{latexsym}
\usepackage{amsfonts,amsmath,amssymb}

\def\pr{\textup{ P\/}}
\def\ex{\textup{E\/}}

\def\eps{\varepsilon}

\def\part{\partial}

\newcommand{\beq}{\begin{equation}}
\newcommand{\eeq}{\end{equation}}

\newtheorem{Theorem}{Theorem}[section]
\newtheorem{Lemma}[Theorem]{Lemma}

\newtheorem{Corollary}[Theorem]{Corollary}

\theoremstyle{remark}

\numberwithin{equation}{section}
\date{\today}
\begin{document}

%\begin{frontmatter}
\title[Stable partitions]{On random stable partitions}

\author{Boris Pittel}
\address{Department of Mathematics, The Ohio State University, Columbus, Ohio 43210, USA}
\email{bgp@math.ohio-state.edu}



\begin{abstract} The stable roommates problem 
does not necessarily have a solution, i.e. a stable matching. We had found that, for the uniformly random instance, the expected number of solutions converges
to $e^{1/2}$ as $n$, the number of members, grows, and with Rob Irving we proved that the 
limiting probability of solvability  is $e^{1/2}/2$, at most. Stephan Mertens's extensive numerics compelled him to conjecture that this probability is of order $n^{-1/4}$.  Jimmy Tan introduced a  notion of a stable cyclic partition, and proved existence of such a partition for every
system of members' preferences, discovering  that presence of odd cycles in a stable partition is equivalent to absence of a stable matching. In this paper we show that the expected number of stable partitions with odd cycles grows as $n^{1/4}$. However the standard deviation of that number is  of order $n^{3/8}\gg n^{1/4}$, too large to conclude
that the odd cycles exist with high probability (whp). Still, as a byproduct, we show 
that whp the fraction of members with more than one stable ``predecessor'' is of order
$n^{-1/4}$. Furthermore, whp the average rank of a predecessor in every stable partition is of
order $n^{1/2}$. The likely size of the largest stable matching is $n/2-O(n^{1/4+o(1)})$,
and  the likely number of pairs of unmatched members blocking the optimal complete matching 
is $O(n^{3/4+o(1)})$.
\end{abstract}

\maketitle

\section{Introduction and main results} 
A  roommates problem instance is specified by an even integer $n$, number of members, and for each $i$ ($1\le i\le n)$ a permutation $\sigma_i$ of the set $[n]=\{1,2,\dots, n\}$ in
which $i$ itself occupies position $n$, ($\sigma_i(n)=i$). The permutation $\sigma_i$ forms the preference list of person $i$: $\sigma_i(k)=j$ if person $j$ occupies position $k$ in the preference list
of person $i$, and each person $i$ is a the end of their own preference list. Equivalently,
the instance can be specified by the ranking list $R_i$ of each person $i$, defined as the inverse permutation of $\sigma_i$: $R_i(j)=k$ if the person $j$ is the $k$-th best for person $i$.

For a given roommates instance with $n$ members, a stable permutation (cyclic partition) is a permutation $\bold\Pi$ of 
$[n]$ such that:
\begin{equation}\label{Tan1,2}
\begin{aligned}
&(1)\, \forall\,i\in [n]: \,\,R_i\bigl(\bold\Pi(i)\bigr)\le R_i\bigl(\bold\Pi^{-1}(i)\bigr);\\
&(2)\, \forall\, 1\le i\neq j\le n:\,\, R_i(j)<R_i\bigl(\bold\Pi^{-1}(i)\bigr)\Longrightarrow R_j(i)>R_j\bigl(\bold\Pi^{-1}(j)\bigr).
\end{aligned}
\end{equation}
Viewing $\bold\Pi$ in terms of its cyclic decomposition, we will refer to $\bold\Pi(i)$ and $\bold\Pi^{-1}(i)$ as the successor of $i$ and the predecessor of $i$ in the permutation $\bold\Pi$. Then condition (1) states no person prefers his predecessor to his successor, and condition (2) states that no two
mutually-unaligned members prefer each other to their predecessors. 
%if $i$ prefers $j$ to his own predecessor, then $j$ prefers his own predecessor to $i$. 
Note that equality in condition (1) is
possible iff $\bold\Pi^2(i)=i$, i.e. either $i$ is a fixed point of $\bold\Pi$, or $(i,\bold\Pi(i))$
is a transposition in $\bold \Pi$---in this case we say that $(i,\bold\Pi(i))$ forms a pair in the
partition $\bold\Pi$. Thus inequality (1) is not vacuous iff $i$ is in a cycle of length $3$ or more, in
which case it is strict. Also if $i$ is a fixed point, then $R_i(\bold\Pi^{-1}(i))=R_i(i)=n$; so condition (2) implies that there are no other fixed points, and every $j\neq i$ prefers his own predecessor to $i$.
Intuitively, each member $i$ proposes to $\bold\Pi(i)$ and holds a proposal from $\bold\Pi^{-1}(i)$. 

Clearly, if a stable partition $\bold\Pi$ is such that it has cycles of length $2$ only, then $\bold\Pi$ is
a stable matching. However, while for every even $n\ge 4$ there are instances without a stable
matching, Tan \cite{Tan}, who introduced the notion of a  cyclic partition $\bold\Pi$, proved that,
for every instance of preferences, (1) there is at least one stable permutation; (2) all stable permutations have the same odd cycles (``parties''); (3) replacing each even cycle $(i_1,i_2,\dots,i_{2m})$ of a stable permutation by the transpositions $(i_1,i_2),\dots,(i_{2m-1},i_{2m})$, or by the transpositions $(i_2,i_3),\dots, (i_{2m},
i_1)$ we get another stable, {\it reduced\/}, permutation;  (4) thus a stable matching exists iff there are no odd cycles. 
%We will focus on stable partitions $\bold\Pi$ without fixed points, whose odd cycles
%have total length between $m_1:=[n^{\sigma}]$, ($\sigma<1/3$),  and $m_2=[n^{1/2}\log n]$. We will call them
%``odd stable'' partitions.
%---we proved in \cite{Pit1} that the stable permutations of the uniformly random instance quite surely do not have a fixed point. Quite surely (q.s.) means with probability $1-O(n^{-K})$, for every $K>0$, a notion introduced in Knuth, Motwani and Pittel 
%\cite{KnuMotPit}.

Suppose that the random problem instance, call it $I_n$, is chosen uniformly at random among all $[(n-1)!]^n$ instances. We showed \cite{Pit2} that the expected number of stable matchings is
$e^{1/2}$ in the limit, implying that the number of stable matchings, if any exist, is bounded in
probability. With Robert Irving \cite{IrvPit} we proved that the probability that a stable matching exists is at most $e^{1/2}/2<1$ in the limit. In a pleasing contrast, the stable partitions do not have a fixed point (odd party of size $1$) with surprisingly high probability $\ge 1- O\bigl(e^{-\sqrt{n}}\bigr)$. So while a stable matching may not exist, stable partitions (that exist always) with high probability
have no ``pariahs'': every member holds a proposal from another member,  while his own proposal is 
accepted by possibly a different  member.

Our task is to analyze asymptotic behavior of a series of leading parameters of
the family of stable (reduced) partitions for $I_n$, and we focus on those that have no fixed point. Among those parameters are 
$\mathcal S_n$ and $\mathcal O_n$, the total number of stable (reduced) partitions and the total number of  ``parties'', i. e. odd (common to all those partitions) cycles. We will prove, for instance, that 
\begin{align}
\ex\bigl[\mathcal S_n\bigr]&=(1+o(1)) \frac{\Gamma(1/4)}{\sqrt{\pi e}\,2^{1/4}}\,n^{1/4},\label{1}\\
\ex\bigl[\mathcal O_n\bigr]&\le(1+o(1)) \frac{\Gamma(1/4)}{4\sqrt{\pi e}\,2^{1/4}}\,n^{1/4}\log n. \label{2}
\end{align}
The fact that $\ex\bigl[\mathcal S_n\bigr]\to\infty$, but at a moderate rate, can be charitably viewed as supporting the claim that, in probability, $\mathcal S_n\to\infty$; if this is the case then 
 with high probability $I_n$ has no stable matching. Numerical experiments conducted
by Stephan Mertens \cite{Mer} made him conjecture that solvability probability goes to zero, as fast
as $n^{-1/4}$. For a rigorous transition from $\ex\bigl[\mathcal S_n\bigr]\to\infty$ to
$\mathcal S_n\to\infty$, one would normally want to show that $\text{Var}(\mathcal S_n)\ll \ex^2\bigl[\mathcal S_n\bigr]$. It turns out, however, that $\text{Var}(\mathcal S_n)$ is of order $n^{3/4}$,
thus exceeding $\ex^2\bigl[\mathcal S_n\bigr]$ by the factor $n^{1/8}$, which invalidates this naive two-moment approach. Can the approach be gainfully modified by narrowing the pool of stable
partitions? 

A key tool for estimating $\text{Var}(\mathcal S_n)$ is an asymptotic formula for
the probability that each of two generic (reduced) partitions $\bold\Pi_1$ and $\bold\Pi_2$ (with the same
odd parties, of course) are stable.  The symmetric difference of the set of matched pairs in $\bold\Pi_1$ and the set of matched pairs in $\bold\Pi_2$ is the edge set of the disjoint even cycles of
length $\ge 4$, whose edges are the matched pairs in $\bold\Pi_1$ interlacing the matched pairs in $\bold\Pi_2$. Each such cycle can be viewed as an even {\it rotation\/} in both partitions, so that
the pair $(\bold\Pi_1,\bold\Pi_2)$ gives rise to $2^{\mu}$ of stable partitions, with $\mu$ being the
total number of those even cycles. Define a random graph $G_n=(\mathcal V_n, \mathcal E_n)$, where $\mathcal V_n$ is the set
of all stable partitions $\bold\Pi$, and $\mathcal E_n$ is the set of pairs $(\bold\Pi_1,\bold\Pi_2)$, each giving
rise to a {\it single\/} even cycle. By \eqref{1}, $\ex[\mathcal V_n]=\ex[\mathcal S_n]$ is of order $n^{1/4}$.
It turns out that $\ex[\mathcal E_n]$ is of order $n^{1/4}$ as well. What, if anything, does this fact tell about the
likely range of $\mathcal S_n$?

There are two positive results that stem from \eqref{1}--\eqref{2}. Tan \cite{Tan}, \cite{Tan1} defined
a maximum stable matching for an instance $I$ as a maximum-size matching $M=M(I)$ which is 
{\it internally\/} stable, i.e. not blocked by any two members from the agent (vertex) set of $M$. He
proved that $|M(I)|=(n-\mathcal O(I))/2$. It follows from \eqref{2} that 
\[
\pr\Biggl(|M(I_n)|\ge \frac{n-\omega(n) n^{1/4}\log n }{2} \Biggr)\ge 1-O(\omega(n)^{-1}) \to 1,
\]
for $\omega(n)\to \infty$, however slowly. In short, the number of members not in the maximum stable
matching is $O_p(n^{1/4}\log n)$. 

Abraham, Bir\'o and Manlove \cite{AbrBirMan} introduced the
alternative notion of  a ``maximally stable'' matching, i.e. a matching $M$ on $[n]$ that is blocked by
the smallest number of pairs, call it $B(I)$, of agents unmatched in $M$. They obtained a two-sided
bound for $B(I)$ in terms of preference lists lengths and the odd cycles. A cruder version of the ABM upper
bound states that $B(I)\le d(I) \mathcal O(I)$, where $d(I)$ is the length of the longest preference list.
Extending our approach, we will show that for the random instance $I_n$, with probability $\ge
1-\exp(-c(\log n)^{2(1+\delta)})$, every member's predecessor is among their best $n^{1/2}(\log n)^{1+\delta} $ choices. So we can apply the last bound with $d(I_n)=n^{1/2}(\log n)^{1+\delta} $. Therefore the bound \eqref{2} together with the ABM bound imply that with high probability there
exists a complete matching which is blocked by $n^{3/4+o(1)}$ pairs, a strikingly small number relative to the total number ($\Theta(n^2)$) of potential blocking pairs.

We will also show that with high probability the sum of the ranks of predecessors in every stable partition
is asymptotic to $n^{3/2}$; consequently the worst predecessor's rank in every stable partition is $n^{1/2}(1-o(1))$ at least, nearly matching $n^{1/2}(\log n)^{1+\delta}$, the likely upper bound. 

Here is an application. Suppose we shrink every member's preference list to their own best $d$ choices. If the constrained instance has no fixed point then neither does the full-lists instance. Consider an instance
$I_{n,d}$  of the stable partition problem chosen uniformly at random among all instances with some $d$ acceptable choices for every member. Randomly, and independently, ordering the remaining $n-1-d$ members for every
member, we will get the uniformly random (full-lists) instance $I_n$. It follows then that if $d\le (1-\eps)n^{1/2}$ ($d\ge n^{1/2}(\log n)^{1+\delta}$ resp.) then with high probability stable partitions for $I_{n,d}$
have (do not have resp.) a fixed point.

Finally, we use the analysis of $\text{Var}(\mathcal S_n)$ to show that the expected fraction of members with multiple stable predecessors is of order $n^{-1/4}$.  

\section{Integral formulas  for stability probabilities}\label{Intforms}
At the core of our proofs are two integral formulas, one for the probability that a generic cyclic
partition is stable, another for the probability that two generic cyclic partitions are stable.
\begin{Lemma}\label{P(partstab)=} Let $\bold\Pi$ be a permutation
of $[n]$ with even cycles of length $2$ only, and possibly a single fixed point $h^*$, i. e. $\bold\Pi(h^*)=
h^*$. Let $\text{Odd}\,(\bold\Pi)$ be the set of all elements from the odd cycles of $\bold\Pi$ with an exception of the fixed point if it is present. Let $D(\bold\Pi)$ be the set of 
unordered pairs $(i\neq j)$ such that $i=\bold\Pi(j)$ or, not exclusively, $j=\bold\Pi(i)$. Then
\begin{equation}\label{p(Pstable)=}
\begin{aligned}
&\pr(\bold\Pi):=\pr\bigl(\bold\Pi\text{ is stable}\bigr)=\idotsint\limits_{\bold x\in [0,1]^{n-1}}F(\bold x)\,d\bold x,\\
F(\bold x)&:=\prod_{h\in \text{Odd}\,(\bold\Pi)}\!\!\!\!\!\!x_h\,\,\cdot\prod_{(i,j)\notin D(\bold\Pi)}\!\!(1-x_ix_j)
\cdot\prod_{k\neq h^*}(1-x_k);
\end{aligned}
\end{equation}
if there is no fixed point $h^*$, then the third product  is replaced by $1$, and $[0,1]^{n-1}$ by $[0,1]^n$.
\end{Lemma}
If  $\bold\Pi$ is  a matching, we get (\cite {Pit2})
\begin{equation}\label{P(Pstable)=}
\pr\bigl(\bold\Pi\text{ is stable}\bigr)=\idotsint\limits_{\bold x\in [0,1]^{n}}\prod_{(i,j)\notin D(\bold\Pi)}\!\!(1-x_ix_j)\,d\bold x.
\end{equation}
\begin{proof} To generate the random instance $I_n$, introduce an array of the independent random variables 
$X_{i,j}$ ($1\le i\neq j\le n$), each distributed uniformly on $[0,1]$. Assume that each member $i\in [n]$
ranks the members $j\neq i$ in increasing order of the variables $X_{i,j}$. Such an ordering is uniform
for every $i$, and the orderings by different members are independent. 
%Let $\bold\Pi$ be a permutation
%of $[n]$ without fixed points, and with even cycles of length $2$ only.  Let $D(\bold\Pi)$ be the set of 
%unordered pairs $(i\neq j)$ such that $i=\bold\Pi(j)$ or, not exclusively, $j=\bold\Pi(i)$.
Then
\begin{multline}\label{P(P)|x)}
\pr\bigl(\bold\Pi\text{ is stable}\,\bold |\,X_{i,\bold\Pi^{-1}(i)}=x_i,\,i\in [n]\bigr)\\
=
\!\!\!\prod_{h\in \text{Odd}(\bold\Pi)}\!\!\!\!\!x_h\,\cdot\prod_{(i,j)\notin D(\bold\Pi)}\!\!(1-x_ix_j)\prod_{k\neq h^*}(1-x_k)
\end{multline}
%where  $\text{Odd}(\bold\Pi)$ is the set of all members from the odd cycles of $\bold\Pi$.
 Indeed,  by \eqref{Tan1,2}, $\bold\Pi$ is stable iff 
\begin{align*}
&(1) \text{ for every }h\in \text{Odd}(\bold\Pi):\,\,X_{h, \bold\Pi(h)}<X_{h,\bold\Pi^{-1}(h)},\\
&(2)\text{ for every }(i,j)\notin D(\bold\Pi),\, i,j\neq h^*:\,\,X_{i,j}<X_{i,\bold\Pi^{-1}(i)}\Rightarrow X_{j,i}>X_{j,\bold\Pi^{-1}(j)},\\
&(3)\text{ for every }i\neq h: X_{i,\bold\Pi^{-1}(i)}<X_{i,h^*}.
\end{align*}
And, conditioned on the event $\bigl\{X_{i,\bold\Pi^{-1}(i)}=x_i,\,i\in [n]\}$, the events above are independent, with (conditional)  probabilities $x_h$,  $1-x_ix_j$ and $1-x_k$ respectively. Using Fubini's theorem, we have
\eqref{p(Pstable)=}.
%\begin{equation}\label{p(Pstable)=}
%\pr(\bold\Pi):=\pr\bigl(\bold\Pi\text{ is stable}\bigr)=\!\overbrace {\idotsint}^{n}_{\bold x\in [0,1]^n}
%\prod_{h\in \text{Odd}(\bold\Pi)}\!\!\!\!\!\!x_h\,\,\cdot\prod_{\{i,j\}\notin D(\bold\Pi)}\!\!(1-x_ix_j)\,d\bold x.
%\end{equation}
\end{proof}
Like analogous formulas in \cite{Pit1}, \cite{Pit2} and \cite{IrvPit}, this is a non-bipartite counterpart  of Knuth's formula for stable bipartite matchings, \cite{Knu}. His derivation was based on the inclusion-exclusion
method, coupled with ingenious observation that the resulting sum equals the multidimensional integral of a product-type integrand resembling our $F(\bold x)$. Of course, we could get a sum-type formula 
for $\pr\bigl(\bold\Pi\text{ is stable}\bigr)$ by expanding the product in \eqref{P(P)|x)} and integrating the
resulting sum term-wise. Moving in the opposite direction, i.e. starting with an inclusion-exclusion formula for $\pr\bigl(\bold\Pi\text{ is stable}\bigr)$, finding an integral-type representation
of the generic summand, and discerning that the sum of the attendant integrands happens to be
an expansion of the ``out-of-the blue'' product in \eqref{P(P)|x)}, would have been very problematic.  
The identity  \eqref{p(Pstable)=} is indispensable for asymptotic estimates,  thanks to a simple, but
powerful, bound
\begin{equation}\label{simple1}
\prod_{\{i,j\}\notin D(\bold\Pi)}\!\!(1-x_ix_j) \le \exp\Bigl(-\frac{s^2}{2}+4.5\Bigr),\quad s:=\sum_{i\in [n]}x_i.
\end{equation}
For instance, this bound and $\prod_k (1-x_k)\le e^{-s}$ will almost immediately yield that  
the stable partitions have no fixed point with  probability $\ge 1- e^{-\Theta(n^{1/2})}$. 
We will prove a surprisingly simple, yet qualitatively sharp estimate: uniformly for a fixed-point free partitions $\bold\Pi$,
\begin{equation}\label{P(Pstab)<simple}
\pr(\bold\Pi\text{ is stable})=O\left(\frac{1}{(n+m-1)!!}\right),\quad m:=|\text{Odd}(\bold\Pi)|.
\end{equation}

We note that Alcalde \cite{Alc} defined an {\it exchange stable\/} matching as a matching $M$ that, to quote from \cite{Man}, ``admits no {\it exchange-blocking pair\/}, which is a pair of members each of whom prefers the other's partner in $M$ to their own''. Cechl\'arov\'a and Manlove \cite{CecMan}
proved that, in stark contrast with the classic stable roommates model, the problem of determining whether a given instance admits an exchange-stable matching is NP-complete. The interested reader
may wish to check that the formula \eqref{P(Pstable)=} continues to hold for $\pr(\bold\Pi\text{ is exchange-stable})$. Consequently the expected number of exchange-stable matchings and the
expected number of the classic stable matchings are exactly the same, implying that the former is also
asymptotic to $e^{1/2}$. Let us call a (fixed-point free) partition $\bold\Pi$ exchange stable if no two members
prefer each other predecessors to their own predecessors under $\bold\Pi$. What about the partitions that are ``doubly-stable'', i.e. stable {\it and\/} exchange stable? It turns out that 
\begin{align*}
&\pr(\bold\Pi):=\pr\bigl(\bold\Pi\text{ is doubly stable}\bigr)=\idotsint\limits_{\bold x\in [0,1]^{n}}F_2(\bold x)\,d\bold x,\\
&\qquad F_2(\bold x):=\prod_{h\in \text{Odd}\,(\bold\Pi)}\!\!\!\!\!\!x_h\,\,\cdot\prod_{(i,j)\notin D(\bold\Pi)}\!\!(1-x_ix_j)^2.
\end{align*}
The counterpart of \eqref{P(Pstab)<simple} is $\pr(\bold\Pi)=O\bigl(2^{-\frac{n+m}{2}}/(n+m-1)!!\bigr)$, implying that the expected number of the doubly stable partitions is of order $2^{-n/2}$, way down from $n^{1/4}$ for the stable partitions.

Continuing, introduce $\mathcal R(\bold\Pi)$, the sum of the ranks of all predecessors in the preference
lists of their successors in a partition $\bold\Pi$. Let $\pr_k(\bold\Pi):=\pr(\bold\Pi\text{ is stable and }\mathcal R(\bold\Pi)=k)$.
\begin{Lemma}\label{P_k(P)=} Suppose $\bold\Pi$ is fixed-point free. Then, letting 
$m:=|\text{Odd}\,(\bold\Pi)|$, and $\bar x=1-x$,
\begin{equation}\label{P_k(P)=}
\begin{aligned}
&\pr_k(\bold\Pi)=\idotsint\limits_{\bold x\in [0,1]^{n}}[z^{k-n-m}] F(\bold x,z)\,d\bold x,\\
F(\bold x,z)&:=\prod_{h\in \text{Odd}\,(\bold\Pi)}\!\!\!\!\!\!x_h\,\,\cdot\prod_{(i,j)\notin D(\bold\Pi)}
\!\!\bigl(\bar x_i\bar x_j+zx_i\bar x_j+z\bar x_ix_j\bigr).
\end{aligned}
\end{equation}
\end{Lemma}
\begin{proof} First of all, using $\chi(A)$ to denote the indicator of an even $A$, we have
\[
\pr_k(\bold\Pi)=[z^k]\ex\Bigl[z^{\mathcal R(\bold\Pi)}\chi(\bold\Pi\text{ is stable})\Bigr].
\]
Here
\begin{align*}
\chi(\bold\Pi\text{ is stable})&=\prod_{(i,j)\notin D(\bold\Pi)}\chi\bigl(X_{i,j}>X_{i,\Pi^{-1}(i)}\text{ or }
X_{j,i}>X_{j,\bold\Pi^{-1}(j)}\bigr)\\
&\times  \prod_{h\in \text{Odd}\,(\bold\Pi)}\chi\bigl(X_{h,\bold\Pi(h)}<
X_{h,\bold\Pi^{-1}(h)}\bigr).
\end{align*}
Furthermore
\begin{align*}
&\mathcal R(\bold\Pi)=\sum_{(i,j)\notin D(\bold\Pi)}\bigl[\chi(X_{i,j}< X_{i,\bold\Pi^{-1}(i)}) +\chi(X_{j,i}<X_{j,\bold\Pi^{-1}(j)})\bigr]\\
&\qquad+\sum_{i\in [n]} 1 +\sum_{h\in \text{Odd}\,(\bold\Pi)}\!\!\chi(X_{h,\bold\Pi(h)}<X_{h,\bold\Pi^{-1}(h)}),
\end{align*}
where the second sum accounts for the pairs $(i,\bold\Pi^{-1}(i))$, $i\in [n]$. So
\begin{align*}
&\ex\Bigl[\left.z^{\mathcal R(\bold\Pi)}\chi(\bold\Pi\text{ is stable})\,\right|\,X_{i,\bold\Pi^{-1}(i)}=x_i,\,i\in [n]
\Bigr]\\
&=z^{n+|\text{Odd}\,(\bold\Pi)|}\!\!\prod_{h\in \text{Odd}\,(\bold\Pi)}\!\!\!\!x_h\\
&\times\!\!\!\prod_{(i,j)\notin D(\bold\Pi)}\!\!
\ex\Bigl[z^{\chi(X_{i,j}< x_i) +\chi(X_{j,i}<x_j)}
\chi(X_{i,j}>x_i\text{ or }X_{j,i}>x_j)\Bigr]\\
&=z^{n+m}\!\!\prod_{h\in \text{Odd}\,(\bold\Pi)}\!\!\!\!x_h\cdot \!\!\!\prod_{(i,j)\notin D(\bold\Pi)}\!\!
\bigl(\bar x_i\bar x_j +z x_i\bar x_j+z\bar x_ix_j\bigr).
\end{align*}
So
\[
\ex\Bigl[z^{\mathcal R(\bold\Pi)}\chi(\bold\Pi\text{ is stable})\Bigr]=z^{n+m}
\idotsint\limits_{\bold x\in [0,1]^{n}} F(\bold x,z)\,d\bold x,
\]
which proves \eqref{P_k(P)=}.
\end{proof}

Finally, suppose we have a pair of distinct cyclic partitions, $\bold\Pi_1$ and $\bold\Pi_2$. Let 
$\pr(\bold\Pi_1,\bold\Pi_2)$ denote the probability that both
$\bold\Pi_1$ and $\bold\Pi_2$ are stable. We assume the two partitions have
the same odd cycles, since otherwise the probability is zero. Suppose also there is no fixed point. Let $\text{Odd}_{1,2}$ stand
for the vertex set of the family of odd cycles, common to both partitions; so $\bold\Pi_1(h)=\bold\Pi_2(h)$ for all
$h\in \text{Odd}_{1,2}$. The cardinality $|\text{Odd}_{1,2}|$ is even, and $\bold\Pi_1$
and $\bold\Pi_2$ induce a pair of perfect matchings $(M_1,M_2)$ on $\text{Even}_{1,2}:=[n]\setminus\text{Odd}_{1,2}$. 
Together, $M_1$ and $M_2$ determine a graph $G(M_1,M_2)
=\Bigl(\text{Even}_{1,2}, E\Bigr)$, with the edge set $E$ formed by the pairs $(i,j)\in M_1\cup M_2$.
Each component of $G(M_1,M_2)$ is either an edge $e\in M_1\cap M_2$,  or a circuit of even length 
at least $4$, in which the edges from $M_1$ and $M_2$ alternate. The edge set for all these 
(alternating) circuits is the symmetric difference $M_1\Delta M_2$.
\begin{Lemma}\label{p(P1,P2stab)=} Let $\pr(\bold\Pi_1,\bold\Pi_2)$ denote the probability that both
$\bold\Pi_1$ and $\bold\Pi_2$ are stable. For $r=1,2$, let $D_r$ be the set of unordered pairs
$(i\neq j)$ such $i=\bold\Pi_r(j)$ or, not exclusively, $j= \bold\Pi_r(i)$. Then
\begin{align*}
&\pr(\bold\Pi_1,\bold\Pi_2)=\idotsint\limits_{\bold x,\bold y\in [0,1]^n} F(\bold x,\bold y)\,d\bold xd\bold y,\\
&F(\bold x,\bold y)=\prod_{h\in \text{Odd}_{1,2}}\!\!\!\!\!x_h\,\,\cdot
\prod_{(i,j)\in D_1^c\cup D_2^c}\!\!\!
[1-x_ix_j-y_iy_j+(x_i\wedge y_i)(x_j\wedge y_j)];
\end{align*}
here 
\[
d\bold x=\prod_{i\in [n]} dx_i,\quad  d\bold y=\prod_{i: \bold\Pi_1(i)\neq \bold\Pi_2(i)} dy_i,
\]
and for every
circuit $\{i_1,\dots,i_{\ell}\}$ of $G(M_1,M_2)$: 
\begin{align*}
&\text{either } \,\,x_{i_1}>y_{i_1},\,x_{i_2}<y_{i_2},\dots, x_{i_{\ell}}<y_{i_{\ell}},\\
&\text{or}\qquad\,\, x_{i_1}<y_{i_1},\,x_{i_2}>y_{i_2},\dots, x_{i_{\ell}}>y_{i_{\ell}}.
\end{align*}
\end{Lemma}
\noindent We omit the proof since it combines the elements of the proof for $\pr(\bold\Pi\text{ is}$
$\text{stable})$ and
of the formula for $\pr(\bold\Pi_1,\bold\Pi_2)$ in the case when $\bold\Pi_1$ and $\bold\Pi_2$ are
matchings, given in \cite{Pit2}. A counterpart of the bound \eqref{simple1}  is 
\begin{equation}\label{simple2}
\begin{aligned}
\prod_{(i,j)\in D_1^c\cup D_2^c}\!\!\!
&[1-x_ix_j-y_iy_j+(x_i\wedge y_i)(x_j\wedge y_j)]\\
&\le e^{2^8}\exp\Bigl(-\frac{s_1^2}{2}-\frac{s_2^2}{2}+\frac{s_{1,2}^2}{2}\Bigr);
\end{aligned}
\end{equation}
here $s_1=\sum_i x_i$, $s_2=\sum_iy_i$, $s_{1,2}=\sum_i (x_i\wedge y_i)$ and $i$ runs over $[n]$.
Never mind enormity of $e^{2^8}$;  like \eqref{simple1}, the bound \eqref{simple2} is both simple and instrumental in
identifying a relatively small, eminently tractable,  part of the integration domain which is ``in charge''
of the asymptotic behavior of $\pr(\bold\Pi_1,\bold\Pi_2)$.

{\bf Note.\/} The reader interested in our prior work on stable roommates problem (\cite{Pit1}, \cite{Pit2}
and \cite{IrvPit}) will not find the inequalities \eqref{simple1} and \eqref{simple2} there. Working on this
project, we detected a technical, estimational, glitch (see the next Section for details) in \cite{Pit1}, equally consequential for analysis in \cite{Pit2} and \cite{IrvPit}. Luckily the new bounds  \eqref{simple1}-\eqref{simple2} allow
to repair this error and, as an unexpected bonus,  to simplify the arguments as well. The analysis in
this paper can be viewed, in part, as a close template for the correction of that embarrassing oversight.
We emphasize that, fortunately, this correction leaves the ultimate asymptotic results in those references
intact.

\begin{proof} As in the proof of Lemma \ref{P(partstab)=}, we use the array $\{X_{i,j}:1\le i\neq j\le n\}$.
By the definition of stability, we have
\[
\{\bold\Pi_1,\,\bold \Pi_2\text{ are both stable}\}=\bigcap_{h\in \text{Odd}(\bold\Pi_{1,2})}\!\!\!\!\!A_h\bigcap\limits_{(i,j)\in D^c_1\cup D^c_2} \bigl(B_{(i,j)}\bigr)^c.
\]
Here $A_h=\bigl\{X_{h,\bold\Pi_{1,2}(h)}<X_{h,\bold\Pi_{1,2}^{-1}(h)}\bigr\}$.
Furthermore:  {\bf (1)\/} if $(i,j)\in D^c_1\cap D^c_2$, then
\begin{align*}
B_{(i,j)}&=\Bigl\{X_{i,j}<X_{i,\bold\Pi_1^{-1}(i)}; \,X_{j,i}<X_{j,\bold\Pi_1^{-1}(j)}\}\\
&\qquad\cup\{X_{i,j}<X_{i,\bold\Pi_2^{-1}(i)}; \,X_{j,i}<X_{j,\bold\Pi_2^{-1}(j)}\Bigr\};
\end{align*}
{\bf (2)\/} if $(i,j)\in  D^c_1\cap D_2$, then necessarily $(i,j)\in M_1^c\cap M_2$, 
and, by stability of $\bold\Pi_1$,
\[
B_{(i,j)}=\Bigl\{X_{i,\bold\Pi_2^{-1}(i)}<X_{i,\bold\Pi_1^{-1}(i)}; \,X_{j,\bold\Pi_2^{-1}(j)}<X_{j,\bold\Pi_1^{-1}(j)}\Bigr\};
\]
{\bf (3)\/} if $(i,j)\in D_1\cap D^c_2$, then necessarily $(i,j)\in M_1\cap M_2^c$ and, by
stability of $\bold\Pi_2$,
\[
B_{(i,j)}=\Bigl\{X_{i,\bold\Pi_1^{-1}(i)}<X_{i,\bold\Pi_2^{-1}(i)}; \,X_{j,\bold\Pi_1^{-1}(j)}<X_{j,\bold\Pi_2^{-1}(j)}\Bigr\}.
\]
Conditioned on the values 
\[
X_{i,\bold\Pi_1^{-1}(i)}=x_i,\,\,(i\in [n]),\quad  X_{i,\bold\Pi_2^{-1}(i)}=y_i, \,\,(i\in [n]:\bold\Pi_1(i)\neq \bold\Pi_2(i)),
\]
the events $A_h$, $B_{(i,j)}$ are all independent. And, denoting the characteristic function of
a set $U\subset [0,1]^{2n}$ by $\chi(U)$, we have $\pr(A_h \boldsymbol|\boldsymbol\cdot) = x_h$,
\[
\pr\bigl((B_{(i,j)})^c \boldsymbol|\boldsymbol\cdot\bigr)=\left\{\begin{aligned}
&1-x_ix_j-y_iy_j+(x_i\wedge y_i)(x_j\wedge y_j),&&\text{ Case }{\bf (1)\/},\\
&\chi(y_i\ge x_i\text{ or } y_j\ge x_j),&&\text{ Case }{\bf (2\/}),\\
&\chi(x_i\ge y_i\text{ or }x_j\ge y_j),&&\text{ Case }{\bf (3)\/}.\end{aligned}\right.
\]
Therefore
\begin{align*}
&\pr\bigl(\bold\Pi_1,\,\bold \Pi_2\text{ are both stable}\boldsymbol|\boldsymbol\cdot\bigr)\\
&\prod_{h\in\text{Odd}_{1,2}}\!\!\!\!\!\!x_h\,\,\,\cdot\prod_{(i,j)\in  D^c_1\cup D^c_2}
\bigl[1-x_ix_j-y_iy_j+(x_i\wedge y_i)(x_j\wedge y_j)\bigr],
\end{align*}
provided that  $\forall\,(i,j)\in M_1^c\cap M_2$, we have $y_i\ge x_i$ or $y_j\ge x_j$ and
$\forall\,(i,j)\in M_1\cap M_2^c$, we have $x_i\ge y_i$ or $x_j\ge y_j$. (The conditional probability is
zero otherwise.) Since the edges from $M_1\Delta M_2$ form the disjoint alternating circuits of length 
$\ge 4$, the condition means that for every such circuit $\{i_1, i_2,\dots, i_{\ell},\}$ [with $(i_1,i_2)\in M_1, (i_2,i_3)\in M_2,\dots, (i_{\ell},i_1)\in M_2$, say], we
have
\begin{align*}
&y_{i_1}\le x_{i_1}\qquad\text{or}\qquad y_{i_2}\le x_{i_2},\\
&y_{i_2}\ge x_{i_2}\qquad\text{or}\qquad y_{i_3}\ge x_{i_3},\\
&\qquad\qquad\qquad\boldsymbol:\\
&y_{i_{\ell-1}}\le x_{i_{\ell-1}}\,\text{ or}\qquad y_{i_\ell}\le x_{i_\ell},\\
&y_{i_{\ell}}\ge x_{i_{\ell}}\qquad\!\!\!\,\,\,\text{ or}\qquad y_{i_1}\ge x_{i_1}.
\end{align*}
We may, of course, assume that all these inequalities are strict. Thus there are only two options on the
circuit: either $x_{i_1}>y_{i_1},\,x_{i_2}<y_{i_2},\,\dots, x_{i_{\ell}}<y_{i_{\ell}},$ or
$x_{i_1}<y_{i_1},\,x_{i_2}>y_{i_2},\,\dots, x_{i_{\ell}}>y_{i_{\ell}}$. (In both options, the inequalities
alternate.)
Application of Fubini's theorem completes the proof.
\end{proof}
\section{Estimation tools}\label{Esttools} To estimate the integrals in Lemma \ref{p(Pstable)=} 
and Lemma \ref{p(P1,P2stab)=} for $n\to\infty$, we will need the following claim, see \cite{Pit0}, \cite{Pit2}:
\begin{Lemma}\label{intervals1} Let $X_1,\dots, X_{\nu}$ be independent $[0,1]$-Uniforms. Let 
$S=\sum_{i\in [\nu]}X_i$ and $\bold V=\{V_i=X_i/S; i\in [\nu]\}$, so that $\sum_{i\in [\nu]}V_i=1$.
Let $\bold L=\{L_i ;\, i\in [\nu]\}$ be the set of lengths of the $\nu$ consecutive subintervals of $[0,1]$ obtained by selecting, independently and uniformly at random, $\nu-1$ points in $[0,1]$. 
Then (with  $\chi(A)$ standing for the indicator of an event $A$) the joint density $f_{S,\bold V}(s,\bold v)$, ($\bold v=(v_1,\dots,v_{\nu-1})$), of $(S,V)$ is given by
\begin{equation}\label{joint<}
\begin{aligned}
f_{S,\bold V}(s,\bold v)&=s^{\nu-1}\chi\bigl(\max_{i\in [\nu]} v_i\le s^{-1}\bigr) \chi(v_1+\cdots+v_{\nu-1}\le 1)\\
&\le \frac{s^{\nu-1}}{(\nu-1)!} f_{\bold L}(\bold v),\quad v_{\nu}:=1-\sum_{i=1}^{\nu-1}v_i;
\end{aligned}
\end{equation}
here $f_{\bold L}(\bold v)=(\nu-1)!\,\chi(v_1+\cdots+v_{\nu-1}\le 1)$ is the density of $(L_1,\dots,L_{\nu-1})$.
\end{Lemma}
We will also use the classic identities,  Andrews, Askey and Roy \cite{AndAskRoy}, Section 1.8:
\begin{equation}\label{int,prod}
\begin{aligned}
&\overbrace {\idotsint}^{\nu}_{\bold x\ge \bold 0 \atop x_1+\cdots+x_{\nu}\le 1}\prod_{i\in [\nu]} x_i^{\alpha_i}\,\,d\bold x=\frac{\prod_{i\in [\nu]}\alpha_i!}{(\nu+\alpha)!},\quad \alpha:=\sum_{i\in [\nu]}\alpha_i,\\
&\overbrace {\idotsint}^{\nu-1}_{\bold x\ge\bold 0\atop x_1+\cdots+x_{\nu}=1}\prod_{i\in [\nu]} x_i^{\alpha_i}\,\,dx_1\cdots dx_{\nu-1} =\frac{\prod_{i\in [\nu]}\alpha_i!}{(\nu-1+\alpha)!}.
\end{aligned}
\end{equation}
%\begin{Lemma}\label{intervals} For $\bold x\in [0,1]^{\nu}$, define $s=\sum_{i\in [\nu]}x_i$ and
%$\bold v=\{v_i=x_i/s;\, i\in [\nu]\}$, so that $\sum_{i\in [\nu]}v_i=1$. Let $\bold L^{(\nu)}=\{L_i^{(\nu)};\, i\in [\nu]\}$ be the set of lengths of the $\nu$ consecutive subintervals of $[0,1]$ obtained by selecting, independently and uniformly at random, $(\nu-1)$ points in $[0,1]$. Then, for $f(s),\,g(\bold v)\ge 0$,
%and $\chi(A)$ standing for the indicator of an event $A$,
%\begin{equation}\label{Int(fg)<}
%\begin{aligned}
%\overbrace {\idotsint}^{\nu}_{\bold x\in [0,1]^{\nu}}\!\! f(s)g(\bold v)\,d\bold x&=\int_0^{\nu}
%\frac{f(s)s^{\nu-1}}{(\nu-1)!}\,\ex\Bigl[g(\bold L^{(\nu)})\chi\bigl(\max_{i\in[\nu]}L_i^{(\nu)}\le s^{-1}\bigr)\Bigr]\,ds\\
%&\le \ex\bigl[g(\bold L^{(\nu)})\bigr]\int_0^{\nu}\frac{f(s)s^{\nu-1}}{(\nu-1)!}\,ds,
%\end{aligned}
%\end{equation}
%the bottom becoming equality if $s\bold v\in [0,1]^{\nu}$ whenever $f(s)g(\bold v)>0$.
%\end{Lemma}
The identity/bound \eqref{joint<} is useful since the random vector $\bold L$ had been well studied. It is known,
for instance, that 
\begin{equation}\label{L^{(nu)}=frac}
\bold L\overset{\mathcal D}\equiv\left\{\frac{w_i}{\sum_{j\in [\nu]}w_j}\right\}_{i\in [\nu]},
\end{equation}
where $w_j$ are independent, exponentially distributed, with the same parameter, $1$ say. Here is this property at work.
\begin{Lemma}\label{sumsofLs} (1) Let $s\ge 2$. If $\eps_{\nu}\to 0$, $\eps_{\nu}\gg
\nu^{-\tfrac{1}{s+1}}$. Then
\begin{equation}\label{sumLj^s}
\pr\Biggl(\Bigl|\frac{\nu^{s-1}}{s!}\sum_{j\in [\nu]}L_j^s-1\Bigr|\ge\eps_{\nu}\Biggr)=
O\Bigl(\exp\bigl(-c\,\eps_{\nu}\nu^{\tfrac{1}{s+1}}\bigr)\Bigl).%\quad c=c(\nu,\eps).
\end{equation}
(2) For $\nu$ even,
\begin{equation}\label{sumsofprods}
\pr\Biggl(\Bigl|2\nu\sum_{j\in [\nu/2]}L_j L_{j+\nu/2}-1\Bigr|\ge\eps_2\Biggr)=
O\Bigl(\exp\bigl(-c\,\eps_{\nu}\nu^{\tfrac{1}{s+1}}\bigr)\Bigl).
\end{equation}
\end{Lemma}
\begin{proof} Observe that $\ex[W]=1$, $\ex[W^s]=s!$. Choose 
\[
a=\left(1+\frac{\eps_{\nu}}{3}\right)\!s!,\quad b=1-\frac{\eps_{\nu}}{3s},
\]
so that $a/b^s<(1+\eps)s!$, for $\nu$ sufficiently large. Then, denoting $\mathcal W^{(\ell)}=\sum_j W_j^{\ell}$,
\begin{align*}
&\pr\Bigl(\nu^{s-1}\sum_{j\in [\nu]}L_j^s\ge (1+\eps_{\nu})s!\Bigr)=
%\pr\left(\frac{\mathcal W^{(2)}}{(\mathcal W^{(1)})^2}\ge \frac{3}{n}\right)
%\!\pr\!\left(\!\frac{\sum_j W_j^2}{\left(\sum_{j'} W_{j'}\right)^2}\ge \frac{3}{n}\right)
\! \pr\!\left(\frac{\mathcal W^{(s)}}{\bigl(W^{(1)}\bigr)^s}\ge \frac{(1+\eps_{\nu})s!}{\nu}\right)\\
&\le \pr\Bigl(W^{(s)}\ge a\nu\text{ or }\mathcal W^{(1)}<b\nu)\Bigr)
\le \!\pr\Bigl(\!\mathcal W^{(s)}\ge a\nu\Bigr) +\pr\Bigl(W^{(1)}<b\nu\Bigr).
\end{align*}
%By Chebyshev's inequality, each of these probabilities is of order $O(n^{-1})$, and then so is 
%$\pr(nT_n\ge 3)$, which proves the part (2). (We note in passing that this probability is, in fact, much
%smaller, certainly below $\exp(-c n^{1/3})$.)
%\\
Since $\ex\bigl[e^{-zW}\bigr]<\infty$ for every $z\ge 0$, the standard application of Chernoff's method yields 
\begin{equation}\label{denom}
\pr\bigl(W^{(1)}<b\nu\bigr) \le \exp(-\nu c(b)),\quad c(b)=b-1-\log b=\Theta\bigl(\eps_{\nu}^2\bigr).
\end{equation}
Bounding  $\pr\Bigl(\!\mathcal W^{(s)}\ge a\nu\Bigr)$ is more problematic since $\ex\bigl[e^{zW^2}\bigr]=\infty$ for $z>0$. Truncation to the rescue! Introduce $V=\min\{W,n^{\alpha}\}$, $(\alpha<1)$; then
\begin{equation}\label{P(WneqV)<}
\pr(W_j\not\equiv V_j,\, j\in [\nu])\le \nu\pr(W\ge \nu^{\alpha})=\nu e^{-\nu^{\alpha}}= e^{-\Theta(\nu^{\alpha})}.
\end{equation}
Further
\begin{align*}
\ex\bigl[e^{n^{-s\alpha}V^s}\bigr]&=\int_0^{n^{\alpha}}e^{(n^{-\alpha} w)^s}\,e^{-w}\,dw +e^{1-\nu^{\alpha}}\\
&\le 1+ \nu^{-s\alpha}\int_0^{\infty}w^s e^{-w}\,dw+O\bigl(\nu^{-2s\alpha}\bigr)\\
%&=1+2n^{-2/3}\int_0^{n^{1/3}}w\,e^{-w}\,dw +O\bigl(n^{-4/3}\bigr)\\
&=1+\nu^{-s\alpha}s!+O\bigl(n^{-2s\alpha}\bigr).
\end{align*}
So
\begin{align*}
\pr\Biggl(\sum_{j\in [\nu]} V_j^s\ge a\nu\Biggr)&\le \frac{\left(1+\nu^{-s\alpha}s!+O\bigl(\nu^{-2s\alpha}\bigr)\right)^{\nu}}
{\exp\bigl(\nu(an^{-s\alpha})\bigr)}\\
&=\exp\Bigl(-\nu^{1-s\alpha}(a-s!) +O\bigl(\nu^{1-2s\alpha}\bigr)\Bigr).
%\\
%&\le\exp\bigl(-0.5(a-s!) \nu^{1-s\alpha}\bigr).
\end{align*}
Combining this bound with \eqref{P(WneqV)<}, we select the best $\alpha=1/(s+1)$ and obtain
\begin{equation}\label{subexp}
\pr\bigl(W^{(s)}\ge a\nu\bigl)\le \exp\bigl(-\hat c\, \eps_{\nu}\nu^{\tfrac{1}{s+1}}\bigr),\quad (\hat c>0).
\end{equation}
This bound combined with \eqref{denom} prove that
\[
\pr\Bigl(\nu^{s-1}\sum_{j\in [\nu]} L_j^s\ge (1+\eps_{\nu})s!\Bigr)=O\Bigl(\exp\bigl(-c\,\eps_{\nu}\nu^{\tfrac{1}{s+1}}\bigr)\Bigl).
\]
Since $\ex\bigl[e^{-zW^s}\bigr]<\infty$ for all $z>0$,  there is no need for truncation, and we get
\[
\pr\Bigl(\nu^{s-1}\sum_{j\in [\nu]} L_j^s\le (1-\eps_{\nu})s!\Bigr) \le e^{-\Theta\bigl(\nu\eps_{\nu}^2\bigr)}, 
\]
So \eqref{sumLj^s} follows. The proof of \eqref{sumsofprods} is similar, and we omit it.
\end{proof}
 %By
%Chebyshev's inequality for the sums of i.i.d. random variables, with finite variance, it easily follows
%from this equi-distribution that, for a fixed $\eps>0$,
%\begin{equation}\label{sumLj^s}
%\begin{aligned}
%\pr\Biggl(\Bigl|\frac{\nu^{s-1}}{s!}\sum_{j\in [\nu]}\bigl(L_j^{(\nu)}\bigr)^s-1\Bigr|\ge\eps\Biggr)&=O(\nu^{-1}),
%\quad (s\ge 2),\\
%\pr\Biggl(\Bigl|2\nu\sum_{j\in [\nu/2]}L_j^{(\nu)}L_{j+\nu/2}^{(\nu)}-1\Bigr|\ge\eps\Biggr)&=O(\nu^{-1}),
%\quad (\nu\text{ even}).
%\end{aligned}
%\end{equation}
 {\bf Note.\/} In \cite{Pit1} we claimed that the probabilities in Lemma \ref{sumsofLs}, for $\eps$ fixed, could be shown to be exponentially small, and used
this claim also in \cite{Pit2} and \cite{IrvPit}, hoping to apply it again in this study. We have realized though that for the right tail of the sums' distributions we could get only  a {\it sub\/}-exponential bound,
see \eqref{subexp}.
Fortunately, with the new inequalities \eqref{simple1}-\eqref{simple2} put to
use, the sub-exponential bounds  \eqref{sumLj^s} and \eqref{sumsofprods} are all we need. The
interested reader may see for themselves that the resulting proof provides a clear recipe for local changes  in \cite{Pit1}, \cite{Pit2} and \cite{IrvPit}, which make the thorny issue of exponential bounds go away completely.

In addition to the bounds \eqref{sumLj^s}, we will need
\begin{equation}\label{P(L^+>)<}
\pr\left(\max_{j\in [\nu]} L_j^{(\nu)}\ge \frac{1.01\log ^2\nu}{\nu}\right)\le e^{-\log^2\nu},
\end{equation}
which directly follows from 
\[
\pr\left(\max_{j\in [\nu]} L_j^{(\nu)}\ge x\right)\le \nu(1-x)^{\nu-1}. 
\]
\section{Estimates of $\ex[\mathcal S_n]$ and $\ex[\mathcal O_n]$, ramifications}
We need to identify a part of the cube $[0,1]^n$ that provides the dominant contribution to
the integral in \eqref{p(Pstable)=}. This will allow us to estimate, sharply,
the expected total length of the odd cycles in the random instance $I_n$. Many of the intermediate estimates can be traced back to \cite{Pit1}, \cite{Pit2} and \cite{IrvPit}. We begin with the pair
of two new, instrumental, bounds for the products in the integrands expressing $\pr(\bold\Pi):=\pr(\bold \Pi\text{ is stable})$ and
$\pr(\bold\Pi_1,\bold\Pi_2):=\pr(\bold\Pi_1\text{ and }\bold\Pi_2\text{ are both stable})$. 

In the statement below and elsewhere we will write $A_n\le_b B_n$ as a shorthand for ``$A_n=O(B_n)$, uniformly over parameters that determine $A_n$, $B_n$'', when the expression for $B_n$ is uncomfortably bulky for an argument of the big O notation. We will also write $A_n\lesssim B_n$ if $\limsup A_n/B_n\le 1$.

\begin{Lemma}\label{simple1,2} 
\begin{align*}
&\prod_{\{i,j\}\notin D(\bold\Pi)}\!\!(1-x_ix_j) \le_b \exp\Bigl(-\frac{s^2}{2}\Bigr),\quad s:=\sum_{i\in [n]}x_i,\\
&\prod_{(i,j)\in D_1^c\cup D_2^c}\!\!\!
[1-x_ix_j-y_iy_j+(x_i\wedge y_i)(x_j\wedge y_j)]\\
&\qquad\qquad\le_b\exp\Bigl(-\frac{s_1^2}{2}-\frac{s_2^2}{2}+\frac{s_{1,2}^2}{2}\Bigr);
\end{align*}
here $s_1=\sum_i x_i$, $s_2=\sum_iy_i$, $s_{1,2}=\sum_i (x_i\wedge y_i)$ and $i$ runs through $[n]$.
\end{Lemma}
\begin{proof}
{\bf (1)\/}  Using $1-z\le e^{-z-z^2/2}$,  we have
\begin{equation}\label{start}
\prod_{(i,j)\notin D(\bold\Pi)}(1-x_ix_j)\le\exp\Biggl(-\sum_{(i,j)\notin D(\bold\Pi)}\Bigl(x_ix_j+\frac{x_i^2x_j^2}{2}\Bigr)\Biggr).
\end{equation}
Here, using $2ab\le a^2+b^2$,
\begin{align}
\sum_{(i,j)\notin D(\bold\Pi)}\!\!\!x_ix_j&=\frac{s^2}{2}-\frac{1}{2}\sum_{i\in [n]}x_i^2-\sum_{i\in [n_1/2]}\!\!x_i
x_{i+n_1/2}-\sum_{h\in \text{Odd}(\bold\Pi)}\!\!\!x_hx_{\bold\Pi(h)}\label{exactexp}\\
&\ge \frac{s^2}{2}-\frac{3}{2}\sum_{i\in [n]}x_i^2.\notag
\end{align}
Analogously, and using $\max_i x_i\le 1$,
\begin{align}
\sum_{(i,j)\notin D(\bold\Pi)}x_i^2x_j^2&\ge \frac{1}{2}\Biggl(\sum_{i\in [n]}x_i^2\Biggr)^2-\frac{3}{2}\sum_{i\in [n]}x_i^4\label{exactexp2}\\
&\ge \frac{1}{2}\Biggl(\sum_{i\in [n]}x_i^2\Biggr)^2-\frac{3}{2}\sum_{i\in [n]}x_i^2.
\end{align}
Therefore
\begin{align}
\sum_{(i,j)\notin D(\bold\Pi)}\left(x_ix_j+\frac{x_i^2x_j^2}{2}\right)&\ge \frac{s^2}{2}+ \frac{1}{2}\left(\sum_{i\in [n]}x_i^2\right)^2-3\sum_{i\in [n]}x_i^2\\
&\ge \frac{s^2}{2}-4.5,
\end{align}
so that 
\begin{equation*}%\label{prod(x)<}
\prod_{(i,j)\notin D(\bold\Pi)}(1-x_ix_j)\le\exp\left(-\frac{s^2}{2}+4.5\right).
\end{equation*}
{\bf (2)\/} Let $M_i$ be the perfect matching on $\text{Even}_{1,2}=[n]\setminus \text{Odd}_{1,2}$, induced by $\bold\Pi_i$. Then $M_1\cap M_2$ is the set of matched pairs common to $\bold\Pi_1$ and
$\bold\Pi_2$, and $M_1\Delta M_2$ is the edge set of the even circuits, of length $4$ at least, formed (in 
alternating fashion) by the matched pairs in $M_1$ and $M_2$.  So $D_1\cup D_2$ is the disjoint
union of $M_1\cap M_2$, $M_1\Delta M_2$ the set of pairs $(i,u_{\bold\Pi_{1,2}})$.

So, given $u_i$, $i\in [n]$,
\begin{equation}\label{sumuiuj}
\begin{aligned}
&\sum_{(i\neq j)\in D^c_1\cap D^c_2}\!\!\!\!\!\!\!u_iu_j=\sum_{(i\neq j)}u_iu_j-\sum_{(i\neq j)\in D_1\cup D_2}
\!\!\!\!\!\!\!u_iu_j\\
&=\sum_{(i\neq j)}u_iu_j\,\,-\sum_{(i,j)\in M_1\cap M_2}\!\!\!\!\!\!\!u_iu_j\,\,-\sum_{(i,j)\in M_1\Delta M_2}
\!\!\!\!\!\!\!u_iu_j\,\,
-\sum_{i\in \text{Odd}_{1,2}}\!\!\!\!\!u_iu_{\bold\Pi_{1,2}(i)}\\
&= \frac{1}{2}\Bigl(\sum_{i\in [n]} u_i\Bigr)^2-\frac{1}{2}\sum_{i\in [n]}u_i^2\,\,-\sum_{(i,j)\in M_1\cap M_2}\!\!\!\!\!\!\!u_iu_j\,\,-\!\sum_{(i,j)\in E_{1,2}}u_iu_j;
%-\sum_{i\in L_{1,2}} u_i^2;
\end{aligned}
\end{equation}
here $E_{1,2}$ is the edge set of the odd cycles and the even circuits, formed by $\bold\Pi_1$ and $\bold\Pi_2$. This exact formula certainly implies that 
\begin{equation*}
\frac{1}{2}\Bigl(\sum_{i\in [n]} u_i\Bigr)^2-3\sum_{i\in [n]}u_i^2\,\,\,\,\le\sum_{(i\neq j)\in D^c_1\cap D^c_2}\!\!\!\!\!\!\!u_iu_j\,\,\,\le\,\,\,\, \frac{1}{2}\Bigl(\sum_{i\in [n]} u_i\Bigr)^2.
\end{equation*}
Therefore we bound
\begin{align*}
&\sum_{(i\neq j)\in D^c_1\cap D^c_2}\bigl[x_ix_j+y_iy_j-(x_i\wedge y_i)(x_j\wedge y_j)\bigr]\\
&\ge\sum_{(i\neq j)}\bigl[x_ix_j+y_iy_j-(x_i\wedge y_i)(x_j\wedge y_j)\bigr]-3\sum_{i\in [n]}(x_i^2+y_i^2)
\\
&\ge\frac{s_1^2}{2}+\frac{s_2^2}{2}-\frac{s_{1.2}^2}{2}-3\sum_{i\in [n]}(x_i^2+y_i^2).
\end{align*}
Furthermore
\begin{align*}
&\bigl[x_ix_j+y_iy_j-(x_i\wedge y_i)(x_j\wedge y_j)\bigr]^2\ge \bigl[x_ix_j+y_iy_j-(x_ix_j\wedge y_iy_j)\bigr]^2\\
&\qquad\qquad\ge\left(\frac{x_ix_j+y_iy_j}{2}\right)^2\ge \frac{1}{8}(x_i^2x_j^2+y_i^2y_j^2).
\end{align*}
So
\begin{align*}
&\sum_{(i\neq j)\in D^c_1\cap D^c_2}\bigl[x_ix_j+y_iy_j-(x_i\wedge y_i)(x_j\wedge y_j)\bigr]^2\\
& \ge\frac{1}{8}\sum_{(i\neq j)\in D^c_1\cap D^c_2}(x_i^2x_j^2+y_i^2y_j^2)\\
&\ge \frac{1}{16}\Biggl(\sum_{i\in[n]}x_i^2\Biggr)^2+ \frac{1}{16}\Biggl(\sum_{i\in[n]}y_i^2\Biggr)^2-\sum_{i\in [n]} \bigl(x_i^4+y_i^4\bigr).
\end{align*}
As 
\[
\sum_{i\in [n]}(x_i^4+y_i^4)\le \sum_{i\in [n]}(x_i^2+y_i^2),
\]
we obtain 
\begin{align*}
&\prod_{(i\neq j)\in D^c_1\cap D^c_2}\!\!\!\!\!\!\!
\bigl[1-x_ix_j-y_iy_j+(x_i\wedge y_i)(x_j\wedge y_j)\bigr]\\
&\le\exp\left(-\frac{s_1^2}{2}-\frac{s_2^2}{2}+\frac{s_{1,2}^2}{2}\right)\\
&\times\exp\Biggl[-\frac{1}{32}\Biggl(\sum_{i\in[n]}x_i^2\Biggr)^2-\frac{1}{32}\Biggl(\sum_{i\in[n]}y_i^2\Biggr)^2%\\
%&\qquad\qquad
+4\sum_{i\in [n]}(x_i^2+y_i^2)\Biggr].
\end{align*}
It remains to observe that $-\frac{z^2}{32}+4z\le 128$.
\end{proof}

\subsection{Bounds for $\pr(\bold\Pi)$, probability of a fixed point, and the likely $|\text{Odd}\,(\bold\Pi)|$}

Here are our first applications of Lemma \ref{simple1,2}.
\begin{Lemma}\label{Odd(P)=m} Denoting $m=|\text{Odd}\,(\bold\Pi)|$,
\[
\pr(\bold\Pi)\le_b\left\{\begin{aligned}
&\frac{e^{-n^{1/2}}}{(n+m-2)!!},&&\bold\Pi\text{ has a fixed point},\\
&\frac{1}{(n+m-1)!!},&&\bold\Pi\text{ has no fixed point}.\end{aligned}\right.
\]
\end{Lemma}
%\begin{Lemma}\label{Odd<m_n} For $m_n:=n^{1/2}\log n$, 
%\[
%P_n:=\pr\bigl(\exists \text{ a stable }\bold\Pi\text{ with }|\text{Odd}\,(\bold\Pi)|>m_n\bigr) \le e^{-0.5\log^2 n}.
%\]
%\end{Lemma}
\begin{proof} For the second case, by \eqref{p(Pstable)=} and Lemma \ref{simple1,2}, 
\[
\pr(\bold\Pi)\le_b \text{Int}(m):=\overbrace {\idotsint}^{n}_{\bold x\in [0,1]^{n}}e^{-\tfrac{s^2}{2}}\!\!\!\!\prod_{h=n-m+1}^n\!\!\!\!\!x_h\, d\bold x.
\]
Disregarding the constraint $\max_i x_i\le 1$, and using \eqref{int,prod},%:\sum_{i\in [n] } x_i\le s}
%\[
%\overbrace {\idotsint}^{n}_{\bold x\in [0,1]^n\atop x_1+\cdots+x_n\le s}\prod_{h=n-m+1}^n\!\!\!\!\!x_h\, d\bold x=\frac{s^{n+m}}{(n+m)!},
%\]
we obtain 
\begin{equation}\label{Dis}
\begin{aligned}
\text{Int}(m)&\le \frac{1}{(n+m-1)!}\int_0^{\infty}e^{-\tfrac{s^2}{2}}s^{n+m-1}\,ds\\
&=\frac{(n+m-2)!!}{(n+m-1)!}=\frac{1}{(n+m-1)!!}.
\end{aligned}
\end{equation}
%A quick glance at the integrand shows that the dominant contribution to the integral comes from $s$
%at the bounded distance from the integrand maximum point $s^*=(n+m-1)^{1/2}$.

If $\bold\Pi$ has a fixed point $h^*$, then using 
\[
\prod_{k\neq h^*}(1-x_k)\le e^{-s},\quad s:=\sum_{k\neq h^*}x_k,
\]
 we obtain that
\[
\pr(\bold\Pi)\le_b\frac{1}{(n+m-2)!}\int_0^{\infty}e^{-s-\frac{s^2}{2}}s^{n+m-2}\,ds.
\]
A quick glance at the integrand shows that the dominant contribution to the integral comes from $s$
 within, say, $\log n$ distance from the integrand's maximum point 
\[
s^*=(n+m-2)^{1/2}-\frac{1}{2}+O\bigl(n^{-1/2}\bigr),
\]
$(n+m-2)^{1/2}$ being the maximum point of $e^{-s^2/2} s^{n+m-2}$. So the above integral is of order
\[
e^{-n^{1/2}}\int_0^{\infty}e^{-\frac{s^2}{2}}s^{n+m-2}\,ds=e^{-n^{1/2}}(n+m-3)!!,
\]
whence
\[
\pr(\bold\Pi\text{ is stable})\le_b\frac{e^{-n^{1/2}}}{(n+m-2)!!}.
\]
%\begin{align*}
%\pr(\bold\Pi\text{ is stable})&\le_b \frac{e^{-n^{1/2}}}{(n+m-2)!}\int_0^{\infty}e^{-\frac{s^2}{2}}s^{n+m-2}\,ds\\
%&=\frac{e^{-n^{1/2}}(n+m-3)!!}{(n+m-2)!}=\frac{e^{-n^{1/2}}}{(n+m-2)!!}.
%\end{align*}
\end{proof}
Now the total number of permutations $\bold\Pi$ of $[n]$ with a fixed point and $|\text{Odd}(\bold\Pi)=
m$ is at most 
\[
n\binom{n-1}{m}m! (n-m-2)!!=\frac{n!}{(n-m-1)!!}.
\]
\begin{Corollary}\label{nofix}
\[
\pr(\text{stable }\bold\Pi\text{'s have a fixed point})=O\bigl(n^2e^{-\sqrt{n}}\bigr)\to 0.
\]
\end{Corollary}
\begin{proof} By Lemma \ref{simple1,2} and the union bound, the probability in question is
of order
\begin{align*}
&e^{-n^{1/2}}\sum_{m\ge 3}\frac{n!}{(n-m-1)!!\,(n+m-2)!!}\\
&\quad\le e^{-n^{1/2}}n \left.\frac{n!}{(n-m-1)!!\,(n+m-2)!!}\right|_{m=3}= O\bigl(n^2e^{-n^{1/2}}\bigr).
\end{align*}
\end{proof}
\noindent 
Our original proof in \cite{Pit2} was considerably more involved, and reliant on the problematic 
existence of the exponential bounds, the issue we touched upon in the previous sections, and will stop
bringing up in the sequel. 

From now on we focus on stable partitions without a fixed point. Here is another low hanging fruit. 
\begin{Corollary}\label{fruit} Denoting by $\text{Odd}\,(\bold\Pi)$ the set of members in the odd cycles
of stable partitions,
\[
\pr\bigl(|\text{Odd}\,(\bold\Pi)|\ge n^{1/2}\log n\bigr)\le_b \exp(-\log^2 n/3),
\]
i.e. with super-polynomially high probability (quite surely in terminology of \cite{KnuMotPit}) the total
length of all odd cycles is below $n^{1/2}\log n$. 
\end{Corollary}
\begin{proof} Denote $m_n=\lceil n^{1/2}\log n\rceil$. The total number of potential stable partitions with an even $|\text{Odd}\,(\bold\Pi)|=m\ge 4$ is at most 
\[
\binom{n}{m}m! (n-m-1)!!=\frac{n!}{(n-m)!!}.
\]
So, by Lemma \ref{simple1,2}, Stirling formula, and the inequality
\[
(1+x)\log(1+x)+(1-x)\log(1-x)\ge x^2,
\]
 the probability in question is of order
\begin{align*}
&\sum_{m=m_n}^n \frac{n!}{(n-m)!!\,(n+m-1)!!}\le_b\sum_{m=m_n}^n\frac{n^n}{(n-m)^{\frac{n-m}{2}}
(n+m)^{\frac{n+m}{2}}}\\
&\le \sum_{m=m_n}^n\exp\Bigl(-\frac{m^2}{2n}\bigr)
\le_b n^{1/2}\!\!\!\!\!\!\!\!\!\int\limits_{x\ge m_n/n^{1/2}}\!\!\!\!\!\!\!e^{-\frac{x^2}{2}}\,dx\ll e^{-\log^2 n/3}.
\end{align*}
\end{proof}
Focusing on the likely stable partitions, we may and will consider only the permutations $\bold\Pi$ 
without a fixed point and with $|\text{Odd}\,(\bold\Pi)|\le m_n$.

\subsection{Sharp estimate of $\pr(\bold\Pi)$}  
In steps,  we will chop off the peripheral parts of the integration cube $[0,1]^n$ till we get to its part narrow enough to allow us to  approximate the integrand in the formula \eqref{p(Pstable)=} within
$1+o(1)$ factor, so that the accumulative error cost is of order $e^{-\Theta(\log^2 n)}$.

{\bf Step 1.\/} For the first reduction, we set $s_n=n^{1/2}+3\log n$, and define
\begin{equation}\label{P_1}
\pr_1(\bold\Pi):= \overbrace {\idotsint}^{n}_{\bold x\in C_1} F(\bold x)\,d\bold x,\quad C_1:=\{\bold x\in [0,1]^n:
s\le s_n\},
\end{equation}
\begin{Lemma}\label{C1}
\begin{equation}\label{P(Pi)-P_1(Pi)}
\pr(\bold\Pi)-\pr_1(\bold\Pi)\le_b \frac{e^{-3\log^2n}}{(n+m-1)!!}.
\end{equation}
\end{Lemma}
\begin{proof} By Lemma \ref{simple1,2}, 
\begin{equation}\label{P-P_1,1}
\pr(\bold\Pi)-\pr_1(\bold\Pi)\le_b \frac{1}{(n+m-1)!}\int\limits_{s\ge s_n}\!\!\exp\left(-\frac{s^2}{2}\right) s^{n+m-1}\,ds.
\end{equation}
The integrand, write it as $e^{h(s)}$, attains its maximum at $s_{n,m}=(n+m-1)^{1/2}$,
and 
\begin{align*}
e^{h(s(n,m))}&=\exp\left(-\frac{n+m-1}{2}\right)(n+m-1)^{\frac{n+m-1}{2}}\\
&\le_b n^{1/2}(n+m-2)!! .
\end{align*}
Further
\begin{align*}
h(s_n)&=h(s(n,m))+(1+o(1))\frac{h^{\prime\prime}(s(n,m))}{2}(s_n-s(n,m))^2\\
&\le h(s(n,m))-4\log^2n,\\
h^\prime(s_n)&=-s_n+\frac{n+m-1}{s_n}\le -5\log n.
\end{align*}
Now, since $h(s)$ is concave, we have 
\[
\int_{s\ge s_n}e^{h(s)}\,ds \le e^{h(s_n)}\int_{s\ge s_n}\exp\bigl(h^\prime(s_n)(s-s_n)\bigr)=
\frac{e^{h(s_n)}}{-h^\prime(s_n)}.
\]
Therefore
\begin{align*}
\pr(\bold\Pi)-\pr_1(\bold\Pi)&\le_b\frac{n^{1/2}e^{-4\log^2n}}{\log n}\frac{(n+m-2)!!}{(n+m-1)!}
\le \frac{e^{-3\log^2n}}{(n+m-1)!!}.
\end{align*}
\end{proof}

Next, motivated by the inequalities  \eqref{exactexp} and \eqref{exactexp2}, we will  derive sharp asymptotics, on progressively smaller $C_j\subset C_1$,
for the leading sums\linebreak $\sum_{i\in [n_1]} x_i^2$,  $\sum_{i\in [n_1/2]}\!x_i
x_{i+n_1/2}$, $(n_1:=n-m)$,
and obtain sufficiently strong upper bounds for the secondary sums  $\sum_{h\in \text{Odd}\,(\bold\Pi)}x_h^2$,  $\sum_{i\in [n]}x_i^4$ and $\sum_{i\in [n]}x_i^6$.
We will end up  with a rather sharp asymptotic formula for $\prod_{(i,j)}(1-x_ix_j)$ on the terminal dominant subset of $C_1$.\\
%We will  use the elements of the proofs above  to derive sharp asymptotics 
%(on most of $C_2$) for the leading sums 
%\[
%\sum_{i\in [n_1]} x_i^2,\quad \sum_{i\in [n_1/2]}\!x_ix_{i+n_1/2},
%\]
%and to obtain sufficiently strong upper bounds for the secondary sums
%\[
%\sum_{h\in \text{Odd}\,(\bold\Pi)}\!\! x_h^2, \quad \sum_{i\in [n]}x_i^4, \quad \sum_{i\in [n]}x_i^6.
%\]
%This will pave the way to a rather sharp bound for the product of the $(1-x_ix_j)$'s on the dominant subset of $C_2$.\\
%Obviously such a simplification is reachedHere is the argument. Having thoroughly used the inequality \eqref{prod(x)<}, we need now its sharper
%version, based again on \eqref{start},\eqref{exactexp} and \eqref{exactexp2}:
%\begin{equation}\label{prod(x)<sharp}
%\begin{aligned}
%&\prod_{(i,j)\notin D(\bold\Pi)}(1-x_ix_j)\le \exp\Biggl(\!-\frac{s^2}{2}+\frac{1}{2}\sum_{i\in [n]}x_i^2+
%\sum_{i\in [n_1/2]}x_ix_{i+n_1/2}\\
%&-\frac{1}{2}\Biggl(\sum_{i\in [n]}x_i^2\Biggr)^2 +O\Biggl(\sum_{h\in \text{Odd}\,(\bold\Pi)}x_h^2\Biggr)
%+O\Biggl(\sum_{i\in [n]}x_i^4\Biggr)\!\!\Biggr).
%\end{aligned}
%\end{equation}

{\bf Step 2.\/} With $s:=\sum_{i\in [n]} x_i$, define $\bold u=\{u_i=x_i/s: i\in [n]\}$. 
Introduce $t_1(\bold u)=\max_{i\in [n]}u_i$. Define 
$
C_2=\Bigl\{\bold x\in C_1: t_1(\bold u)\le 1.01\frac{\log^2n}{n}\Bigr\},
$
and let $P_j(\bold\Pi)$ be the  integral of $F(\bold x)$ over $C_j$. Introduce $L_1,\dots,L_{n}$, the lengths of the $n$ consecutive subintervals of $[0,1]$ obtained by choosing, at random, $n-1$ points in $[0,1]$.  Applying  
Lemma \ref{intervals1}, the identity \eqref{int,prod} and 
Lemma \ref{sumsofLs} (1) with $\nu=n$, we have
\begin{align*}
&P_1(\bold\Pi)-P_2(\bold\Pi)\le\!\!\overbrace {\idotsint}^{n}_{\bold x\ge \bold 0}\! s^m e^{-\frac{s^2}{2}}\chi\!\left\{t_1(\bold u)\ge 1.01\frac{\log^2n}{n}\right\}\!
%\prod_{h=n-m+1}^n\!\!\!\!\! u_h \,\,d\bold x\notag\\
\prod_{h\in\text{Odd}(\bold\Pi)}\!\!\!\!\! u_h \,\,d\bold x\notag\\
&\le \frac{\ex\left[\chi\left\{\max_{i\in [n]} L_i\ge 1.01\frac{\log^2n}{n}\right\}\prod_{h=1}^m L_h\right]}
{(n-1)!}\int_0^{\infty}\!\! e^{-\frac{s^2}{2}}s^{m+n-1}\,ds.
%\iint\limits_{\eta,\,\xi\ge 0}e^{-\frac{(\eta+\xi)^2}{2}}\eta^{n_1-1}\xi^{2m-1}\,d\eta\, d\xi\notag\\
%&\le \frac{e^{-\Theta(n^{1/3-\sigma})}}{(n+m-1)!}\int_0^{\infty}e^{-\frac{s^2}{2}} s^{n+m-1}
%=\frac{e^{-\Theta(n^{1/3-\sigma})}}{(n+m-1)!!}.
\end{align*}
By the union bound, the expected value is below
\begin{multline}\label{union}
n\ex\left[\chi\left\{ L_n\ge 1.01\frac{\log^2n}{n}\right\}\prod_{h=1}^{m-1} L_h\right]\le
 n!\overbrace {\idotsint}^{n-1}_{z_1+\cdots+z_{n-1}\atop\le 1-1.01\frac{\log^2n}{n}}\prod_{h=1}^{m-1} z_h\,d\bold z\\
= n!\left(1-1.01\frac{\log^2n}{n}\right)^{n+m-2}\overbrace {\idotsint}^{n-1}_{z_1+\cdots+z_{n-1}\le 1}\prod_{h=1}^{m-1} z_h\,d\bold z\le \frac{n! e^{-1.01\log^2 n}}{(n+m-2)!}.
\end{multline}
\begin{Lemma}\label{P1-P2}
\begin{equation}\label{P_1(Pi)-P_2(Pi)}
P_1(\bold\Pi)-P_2(\bold\Pi)\le \frac{e^{-\log^2 n}}{(n+m-1)!!}.
\end{equation}
\end{Lemma}
In addition, since $m\le m_n =\lceil n^{1/2}\log n\rceil$ and $s\le s_n=n^{1/2}+3\log n$, it follows from \eqref{P_1(Pi)-P_2(Pi)} that on $C_2$
\begin{equation}\label{sumx_h^2}
\begin{aligned}
&\sum_{h\in \text{Odd}\,(\bold\Pi)}x_h\le_b s n^{-1}\log^3 n,\quad 
\sum_{h\in \text{Odd}\,(\bold\Pi)}x_h^2 \le_b n^{-1/2}\log^5 n,\\
&\qquad\quad\sum_{i=1}^n x_i^{4}\le_b n^{-1}\log^8n, \quad \sum_{i=1}^n x_i^{6}\le_bn^{-2}\log^{12}n.
\end{aligned}
\end{equation}

{\bf Step 3.\/} With $\xi:=\sum_{i\in [n_1]}x_i$, define $\bold v=\{v_i=x_i/\xi: i\in [n_1]\}$. 
Introduce $t_2(\bold v)=\sum_{i\in [n_1]}\!v_i^2$. Define 
\[
C_3=\left\{\bold x\in C_2: \Bigl|\frac{n_1}{2}t_2(\bold v)-1\Bigr|\le n^{-\sigma}\right\},\quad \sigma<1/3.
\]
Introduce $\mathcal L_1,\dots, \mathcal L_{n_1}$, the lengths of the $n_1$ consecutive subintervals of $[0,1]$ in the partition of $[0,1]$
by the random $n_1-1$ points.  Analogously to {\bf Step 2\/},
%Applying  
%Lemma \ref{intervals1}, the identity \eqref{int,prod} and 
%Lemma \ref{sumsofLs} (1) with $\nu=n_1$, 
we have
\begin{equation*}\label{P_2(Pi)-P_3(Pi)}
\begin{aligned}
&P_2(\bold\Pi)-P_3(\bold\Pi)\le\!\!\overbrace {\idotsint}^{n}_{\bold x\ge \bold 0}\!e^{-\frac{s^2}{2}}\chi\!\left\{\!\Bigl|\frac{n_1}{2}t_2(\bold v)-1\Bigr|\ge n^{-\sigma}\!\right\}\!
%\prod_{h=n-m+1}^n\!\!\!\!\! x_h \,\,d\bold x\\
\prod_{h\in\text{Odd}(\bold\Pi)}\!\!\!\!\!\!\! u_h \,\,d\bold x\\
&\le \frac{\pr\left(\Bigl|\frac{n_1}{2}t_2(\bold{\mathcal  L})-1\Bigr|\ge n^{-\sigma}\right)}
{(n_1-1)!(2m-1)!}\iint\limits_{\eta,\,\xi\ge 0}e^{-\frac{(\xi+\eta)^2}{2}}\xi^{n_1-1}\eta^{2m-1}\,d\xi\, d\eta\\
&\le \frac{e^{-\Theta(n^{1/3-\sigma})}}{(n+m-1)!}\int_0^{\infty}e^{-\frac{s^2}{2}} s^{n+m-1}\,ds
=\frac{e^{-\Theta(n^{1/3-\sigma})}}{(n+m-1)!!}.
\end{aligned}
\end{equation*}
\begin{Lemma}\label{P2-P3}
\begin{equation}\label{P_2(Pi)-P_3(Pi)}
P_2(\bold\Pi)-P_3(\bold\Pi)\le \frac{e^{-\Theta(n^{1/3-\sigma})}}{(n+m-1)!!}.
\end{equation}
\end{Lemma}
\noindent Similarly, with $t_3(\bold v):=\sum_{i\in [n_1/2]}\,v_iv_{i+n_1/2}$ and
\[
C_4:=\left\{\bold x\in C_3:\left|2n_1t_3(\bold v)-1\right|\le n^{-\sigma}\right\},
\]
\begin{Lemma}\label{P3-P4}
\begin{equation}\label{P_3(Pi)-P_4(Pi)}
P_3(\bold\Pi)-P_4(\bold\Pi)\le\frac{e^{-\Theta(n^{1/3-\sigma})}}{(n+m-1)!!}.
\end{equation}
\end{Lemma}

Combining the estimates \eqref{P(Pi)-P_1(Pi)}, \eqref{P_1(Pi)-P_2(Pi)}, \eqref{P_2(Pi)-P_3(Pi)},
we have 
\begin{Lemma}\label{summary} Let $\bold\Pi$ be such that $m=|\text{Odd}\,(\bold\Pi)|\le m_n$.
%, (1) with $P_3(\bold\Pi)$
%defined as the integral of $F(\bold x)$ over $C_3$,
Then 
\begin{equation*}
P(\bold\Pi)-P_4(\bold\Pi)\le \frac{e^{-\Theta(\log^2 n)}}{(n+m-1)!!},
\end{equation*}
where $P_4(\bold\Pi)$ is the integral of $F(\bold x)$ over $C_4\subset [0,1]^n$ defined by the
additional constraints:  with $s:=\sum_{i\in [n]}x_i$, $s_n=n^{1/2}+3\log n$, $\xi=\sum_{i\in [n_1]}x_i$,
%$\eta:=\sum_{h}x_h$,  $(h\in \text{Odd}\,(\bold\Pi))$,
\begin{align}
&\qquad\qquad s\le s_n,\quad \max_{i\in [n]} x_i\le 1.02\, \frac{s\log^2 n}{n};\label{s<,max x_i<}\\
%&\Bigl|\frac{n_1}{2}t_1(\bold v)-1\Bigr|\le n^{-\sigma}, \,\,
&\left|\frac{n_1\sum_{i\in [n_1]}x_i^2}{2\xi^2}-1\right|\le n^{-\sigma}, \,
\,\left|\frac{2n_1\sum_{i\in [n_1/2]}x_i x_{i+n_1/2}}{\xi^2}-1\right|\le n^{-\sigma}.\label{sumu_iu_{i+n1/2}/(sumu_j)^2}
\end{align}
\end{Lemma}
\noindent The constraint \eqref{sumu_iu_{i+n1/2}/(sumu_j)^2} involves only $\{x_i\}_{i\in [n_1]}$. Furthermore, 
given $s$, the constraint  \eqref{s<,max x_i<} imposes the uniform upper bound
for the {\it individual\/} components $x_i$, $i\in [n]$: no mixing the components $x_i$, $i\in [n_1]$,
and $x_h$, $h\in\text{Odd}\,(\bold\Pi)$, either. Also, this constraint  {\it implies\/} that 
\begin{equation}\label{max<}
\max_{i\in n]}x_i\le 4 n^{-1/2}\log^2 n =o(1),
\end{equation}
meaning that the constraint $\max_i x_i\le 1$ is superfluous. Moreover,  the inequality \eqref{max<}
yields  the  {\it equality\/} 
\[
\prod_{(i,j)\notin D(\bold\Pi)}\!(1-x_i x_j)=\exp\Biggl(\!-\sum_{(i,j)\notin D(\bold\Pi)}\!\Bigl(x_ix_j+\frac{x_i^2x_j^2}{2}
\Bigr)+O\Bigl(\sum_{i\in [n]} x_i^6\Bigr)\!\Biggr),
\]
that holds uniformly for $\bold x\in C_4$, with the remainder term $\ll n^{-1}$, see \eqref{sumx_h^2}.
It is the matter of simple algebra
to obtain from the constraints on $C_4$: 
\begin{Lemma}\label{asform} Uniformly for $m\le m_n$ and $\bold x\in C_4$,
\begin{equation}\label{F=sharp}
F(\bold x)=\exp\!\left(\!-\frac{s^2}{2}\left(\!1-\frac{3}{n}\!\right)-\frac{s^4}{n^2}+O(n^{-\sigma})\!\right)\!
\prod_{h\in \text{Odd}\,(\bold\Pi)}\!\!\!\! x_h.
\end{equation}
\end{Lemma}
Thus, introducing $\eta=\sum_{h\in \text{Odd}\,(\bold\Pi)}x_h$, so that as $s=\xi+\eta$, 
within the factor $1+O(n^{-\sigma})$ the integrand depends on $(\xi,\eta)$
and $\prod_h x_h$. Observe also that, on $C_4$,
\[
\max_{i\in [n_1]}\frac{x_i}{\xi}\sim\max_{i\in [n_1]}\frac{x_i}{s}\le 1.02\frac{\log^2 n}{n}\ll \frac{1}{\xi}.
\]
So denoting $\psi_n(s)=\tfrac{s^2}{2}\bigl(1-\tfrac{3}{n}\bigr)+\tfrac{s^4}{n^2}$, and applying Lemma \ref{intervals1}, \eqref{joint<}, we have: $P_4(\bold\Pi)$, the integral of $F(\bold x)$ over $C_4$, is given by
%($u_h=x_h/\xi$). In addition, we notice that with the {\it explicit\/} constraint $\max x_i\le 1$ gone, the constrains \eqref{t1,2,3(v)} and \eqref{t4(v),t5(u)} contain only  $\bold v$
% $\bold u$, while
%the top constrain \eqref{eta,xi} contain $\eta$ and $\xi$ only. Besides, by \eqref{max<}, on $C_3$
%\begin{equation}\label{max<1/eta,1/xi}
%\max_i v_i <\frac{1}{\eta},\quad \max_h u_h<\frac{1}{\xi}.
%\end{equation}
%Using \eqref{max<1/eta,1/xi} in the equality  \eqref{joint<}, we have then:  on $C_3$, the joint density $f(\eta ,\bold v; \xi,\bold u)$ of 
%\[
%\mathcal S_1:=\sum_{i\in [n_1]}X_i,\,\,\bold V:=\left\{\tfrac{X_i}{\mathcal S_1}\right\};\quad
%\mathcal S_2:=\sum_{h=1}^{n-m+1}X_h,\,\, \bold U:=\left\{\tfrac{X_h}{\mathcal S_2}\right\}
%\]
%is  given by 
%\[
%f(\eta,\bold v; \xi,\bold u)=\frac{\eta^{n_1-1}}{(n_1-1)!}  f(\bold v)\times
%\frac{\xi^{m-1}}{(m-1)!} f(\bold u).
%\]
%Here 
%\[
%f(\bold v)=(n_1-1)!\,\chi\left\{\sum_{i=1}^{n_1-1}v_i\le 1\right\},\quad f(\bold u)=(m-1)!\,
%\chi\left\{\sum_{j=1}^{m-1}u_j\le 1\right\}
%\]
%are the joint densities of the first intervals $L_1,\dots,L_{n_1-1}$ and the first intervals $\mathcal L_1,\dots,
%\mathcal L_{m-1}$, in the  partitions of $[0,1]$ by the $n_1-1$ random points and the $m-1$ random points respectively.
\begin{equation}\label{P_3(P)approx}
\begin{aligned}
&P_4(\bold\Pi)=\bigl(1+O(n^{-\sigma})\bigr)\!\!\overbrace {\idotsint}^{n}_{\bold x\in C_4}
e^{-\psi_n(\xi+\eta)}\prod_{h\in \text{Odd}\,(\bold\Pi)} x_h\,d\bold x\\
&=\bigl(1+O(n^{-\sigma})\bigr)\!\!\iint\limits_{\xi+\eta\le s_n}\!\!e^{-\psi_n(\eta+\xi)}\frac{\xi^{n_1-1}}{(n_1-1)!}\cdot \frac{\eta^{2m-1}}{(m-1)!}\,d\eta\,d\xi\\
&\cdot\pr\Biggl(\Bigl|\frac{n_1}{2}\!\sum_{i\in [n_1]}\mathcal L_i^2-1\Bigr|\le n^{-\sigma},\,\,
\Bigl|2n_1\!\!\sum_{i\in [n_1/2]}\mathcal L_i\mathcal L_{i+n_1/2}-1\Bigr|\le n^{-\sigma}\Biggr).
%\\
%&\qquad\times\pr\Bigl(\!\{t_j(\bold L)\}_{1\le j\le 4}\text{ meet }\eqref{t1,2,3(v)}, \eqref{t4(v),t5(u)}\!\Bigr)\\
%&\qquad\times\ex\!\left[\chi(t_5(\boldsymbol{\mathcal L})\text{ meet }\eqref{t4(v),t5(u)})\,\,\cdot\!\!\prod_{h=n-m+1}^n\!\!\!\!\! \mathcal L_h\right]
\end{aligned}
\end{equation}
From the step {\bf (4)\/} we know that the probability factor is at least $1-e^{-\log ^2n}$.
% and the
%expected value is at least  
%\begin{align*}
%&\ex\Biggl[\prod_{k\in [m]} \mathcal L_k\Biggr] -\frac{m!}{(2m-1)!}\,e^{-\log^2 n}\\
%&=(m-1)!\!\!\!\overbrace{\idotsint}^{m-1}_{u_1+\cdots+u_m=1} \prod_{k=1}^{m} u_k \,\,du_1
%\cdots du_{m-1}-\frac{m!}{(2m-1)!}\,e^{-\log^2 n}\\
%&=\frac{(m-1)!}{(2m-1)!} -\frac{m!}{(2m-1)!}\,e^{-\log^2 n}.
%\end{align*}
%So the product of the two factors is 
%\[
%\frac{(m-1)!}{(2m-1)!}\bigl(1-e^{-\Theta(\log^2n)}\bigr).
%\]
The double integral, denote it $ I_{n,m}$, is given by
\[
I_{n,m}=\frac{(2m-1)!}{(m-1)!\,(n+m-1)!}\int_{s\le s_n} e^{-\psi_n(s)} s^{n+m-1}\,ds.
\]
The integrand attains its maximum at  $\hat s =(n+m-1)^{1/2}-\Theta(n^{-1/2})$, so that
$s_n-\hat s\ge 2\log n$. 
and it is easy to show that
\[
\int\limits_{|s-\hat s|\ge \log n}\!\!\!\!\!\!e^{-\psi_n(s)}s^{n+m-1}\,ds\le e^{-\Theta(\log ^2n)}
\!\!\!\!\!\!\!\int\limits_{|s-s^*|\le\log n}\!\!\!\!\!\!\!e^{-\psi_n(a)}s^{n+m-1}\,ds.
\]
Besides, $s^4/n^2=1+O(m/n)$ for $|s-\hat s|\le \log n$.  Therefore
\begin{align*}
I_{n,m}&=e^{-1+O(m/n)}\frac{(2m-1)!}{(m-1)!\,(n+m-1)!}\int_{s\ge 0}e^{-\frac{s^2(1-3/n)}{2}}s^{n+m-1}\,ds\\
&=e^{-1+O(m/n)}\frac{(2m-1)!}{(m-1)!\,(n+m-1)!}\cdot\frac{(n+m-2)!!}{(1-3/n)^{(n+m)/2}}\\
&=e^{1/2+O(m/n)}\frac{(2m-1)!}{(m-1)!\,(n+m-1)!!}. 
\end{align*}
%So the double integral in \eqref{P_3(P)approx} also equals
%\[
%e^{1/2+O(m/n)}\,\frac{(2m-1)!}{(m-1)!\,(n+m-1)!!}.
%\]
Since $m/n=O(n^{-1/2}\log n)$, and $\sigma<1/3$ in \eqref{P_3(P)approx}, we have proved
\begin{Lemma}\label{P_4(P)=exact} Uniformly for even $m\le m_n$ and $\bold\Pi$ with $|\text{Odd}\,(\bold\Pi)|=m$,
\[
P_4(\bold\Pi)=\bigl(1+O(n^{-\sigma})\bigr)\frac{e^{1/2}}{(n+m-1)!!}.
\]
Consequently, by Lemma \ref{summary} ,
\begin{equation}\label{P(P)sim!}
P(\bold\Pi)=\bigl(1+O(n^{-\sigma})\bigr)\frac{e^{1/2}}{(n+m-1)!!}.
\end{equation}
\end{Lemma}
\noindent {\bf Note.\/} The formula \eqref{P(P)sim!} works for $m=0$ as well, meaning that 
\[
\pr(\text{matching }\bold\Pi\text{ is stable})=\bigl(1+O(n^{-\sigma})\bigr)\frac{e^{1/2}}{(n-1)!!}. 
\]
So the expected number of stable {\it matchings\/} tends to $e^{1/2}$ as $n\to\infty$, \cite{Pit1}.

\subsection{The expectations of the numbers of stable partitions and odd
parties} 
\begin{Theorem}\label{E[On]} Let $\mathcal S_n$ and $\mathcal O_n$ denote the total number of odd
stable partitions $\bold\Pi$, and the total number of odd cycles.  Then
\begin{align}
\ex\bigl[\mathcal S_n\bigr]
%&=\bigl(1+O(n^{-\sigma})\bigr)\sqrt{\frac{2}{\pi e}}\sum_{m\in [m_1,m_2]}\!\!\!\!\exp\left(-\frac{m^2}{2n}\right) m^{-1/2}\\
&=\bigl(1+O(n^{-1/4})\bigr)\frac{\Gamma(1/4)}{\sqrt{\pi e}\, 2^{1/4}}\,n^{1/4},\label{ESn=}\\
\ex\bigl[\mathcal O_n\bigr]&\lesssim \frac{\Gamma(1/4)}{4\sqrt{\pi e}\, 2^{1/4}}\,n^{1/4}\log n.\label{EOn<}
\end{align}
\end{Theorem}
\begin{proof}  For even $m$, let $f(m)$ denote the total number of permutations of $m$, having only
odd cycles, each of length $3$ at least. For even $k$, let
$f(m,k)$ denote the total number of permutations of $[m]$ having only $k$ odd cycles, each of length 
$3$, at least;  so $f(m)=\sum_k f(m,k)$. Then the total number of permutations of $[n]$ with $k$ odd cycles, each of length $3$ at least, with $m$ elements overall, and even cycles of length $2$ only, is 
$\binom{n}{m}f(m,k)(n-m-1)!!$. So, by
Lemma \ref{P_4(P)=exact}, we have
\begin{equation}\label{E[Sn]sim}
\ex\bigl[\mathcal S_n\bigr]=\bigl(e^{1/2}+O(n^{-\sigma})\bigr)\sum_{m\le m_n}\!\!\binom{n}{m}
\frac{f(m)(n-m-1)!!}{(n+m-1)!!}.
\end{equation}
A standard argument from permutation enumeration shows that 
\begin{equation}\label{genf(m)}
\sum_{m\ge 4}\frac{f(m)}{m!}\,x^m =\exp\!\!\left(\sum_{\text{odd }j\ge 3}\frac{x^j}{j}\right)
=e^{-x}\sqrt{\frac{1+x}{1-x}}, \quad (|x|<1).
\end{equation}
So, using the saddle-point method (Flajolet and Sedgewick \cite{FlaSed}),
\begin{equation}\label{f(m)sim}
f(m)=\bigl(e^{-1}\sqrt{2}+O(m^{-1})\bigr)\frac{(2m-1)!!}{2^m}.
\end{equation}
%Appearance of $O(n^{-\sigma})$ in the outside factor in \eqref{E[Sn]sim} and of $O(m^{-1})$ in the
%last formula explains why we had to choose $m_1\ge n^{\sigma}$. 
With a bit of work, based on Stirling
formula, it follows that
\begin{align*}
&\bigl(e^{1/2}+O(n^{-\sigma})\bigr)\binom{n}{m}\frac{f(m)(n-m-1)!!}{(n+m-1)!!}\\
&=\bigl(1+O(n^{-\sigma}+m^{-1})\bigr)\sqrt{\frac{2}{\pi e}}\cdot m^{-1/2}\exp\left(-\frac{m^2}{2n}\right).
\end{align*}
Combining this formula with \eqref{E[Sn]sim}, and choosing $\sigma=1/4$, we complete the proof of \eqref{ESn=}.

A bivariate extension of \eqref{genf(m)} is
\begin{equation*}
\sum_{m\ge 4}\frac{x^m}{m!} \sum_{k\ge 2}y^kf(m,k)=\exp\!\!\left(y\!\sum_{\text{odd }j\ge 3}\frac{x^j}{j}\right).
\end{equation*}
Differentiating this identity at $y=1$, we obtain
\begin{multline}\label{gensumkf(m,k)}
\sum_{m\ge 4}\frac{x^m}{m!} \sum_{k\ge 2}kf(m,k)=
\sum_{\text{odd }j\ge 3}\frac{x^j}{j}\exp\left(\sum_{\text{odd }j\ge 3}\frac{x^j}{j}\right)\\
=\left(\frac{1}{2}\log\frac{1}{1-x}+\frac{1}{2}\log(1+x)-1\right) e^{-x}\sqrt{\frac{1+x}{1-x}}.
\end{multline}
So, analogously to \eqref{f(m)sim}, we obtain
\begin{equation*}
\sum_{k\ge 2}kf(m,k)=\bigl(1+O(m^{-1})\bigr)\frac{e^{-1}\sqrt{2}\log m}{2}\cdot\frac{(2m-1)!!}{2^m}.
\end{equation*}
Combining this formula with the counterpart of \eqref{E[Sn]sim}, i.e. with
\begin{equation*}
\ex\bigl[\mathcal O_n\bigr]\le \bigl(e^{1/2}+O(n^{-\sigma})\bigr)\sum_{m\le m_n}\!\!\binom{n}{m}
\frac{(n-m)!!}{(n+m-1)!!}\sum_{k\ge 2}kf(m,k),
\end{equation*}
we have \eqref{EOn<}.
\end{proof}

Tan \cite{Tan}, \cite{Tan1} defined a maximum stable matching for an instance $I$ as a matching $M$ of maximum size
(number of matched pairs) such that no pair of members, both having a partner in $M$, prefer
each other to their partners. In short, no two members assigned in $M$, but not to each
other, block $M$. He proved that a maximum stable matching has size $(n-\mathcal O)/2$, (see also
Manlove \cite{Man}).

\begin{Corollary}\label{maxst} Let $\mathcal M_n$ denote the size of the maximum stable matching
for the random instance $I_n$. Then
\begin{align*}
&\quad\ex\bigl[\mathcal M_n\bigr]\ge \frac{n - c n^{1/4}\log n}{2},\quad c=\frac{\Gamma(1/4)}{3\sqrt{\pi e}\, 2^{1/4}},\\
&\pr\Bigl(\mathcal M_n\ge \frac{n - n^{1/4}\log^2 n}{2}\Bigr)\ge 1- O\bigl(\log ^{-1} n)\bigr),
\end{align*}
so that the number of members unassigned in the maximum stable matching is likely to be of order
$O\bigl(n^{1/4}\log ^2n)$.
\end{Corollary}

\subsection{A ``maximally stable'' matching in the random instance $I_n$} For a given set of preferences, Abraham, Bir\'o and Manlove
\cite{AbrBirMan} (see also \cite{Man})
defined a ``maximally stable'' matching as a perfect matching $M$  on $[n]$ that is blocked by the smallest number of pairs, $B(I_n)$, of members not matched with each other in $M$. (Two members block $M$
if they prefer each other to their partners in $M$.)  A weaker corollary of the bound in \cite{AbrBirMan}
states that $B(I_n)\le d(I_n)\mathcal O(I_n)$, where $O(I_n)$ is the number of odd parties (common to
all stable partitions for $I_n$) and $d(I_n)$ is the length of the longest preference list.  

Once we estimate $R_{\text{max}}$, defined as the largest rank of a predecessor in the uniformly random instance $I_n$, we will be able to apply the ABM inequality via replacing $d(I_n)$ with $R_{\text{max}}$.

For a stable $\bold\Pi$ (without a fixed point), introduce $X(\bold\Pi):=\max_i X_{i,\bold\Pi^{-1}(i)}$.
Intuitively, $\max_{\bold\Pi} X(\bold\Pi)$ controls the worst predecessor's rank. From Lemma \ref{summary}, and the proof of Theorem \ref{E[On]}, it follows that 
\begin{equation*}
\pr\Biggl(\max_{\bold\Pi}X(\bold\Pi)\ge \frac{\log^2n}{n^{1/2}}\Biggr)\le e^{-\Theta(\log^2 n)}.
\end{equation*}
A bit more generally, for every $\delta>0$,
\begin{equation}\label{P_delta}
\pr_{\delta}:=\pr\Biggl(\max_{\bold\Pi}X(\bold\Pi)\ge \frac{\log^{1+\delta}n}{n^{1/2}}\Biggr)\le e^{-\Theta(\log^{1+\delta} n)}.
\end{equation}
Denoting $x_n=\frac{\log^{1+\delta}n}{n^{1/2}}$, let $R_i:=|\{j\neq i\,:\, X_{i,j}\le x_n\}|$. Since $X_{i,j}$
are independent $[0,1]$-Uniforms, we have $R_i\overset{\mathcal D}\equiv\text{Bin}(n-1, p=x_n)$. Let $c>1$; by the classic (Chernoff) bound for the tails of the binomial distribution,
\[
\pr(R_i\ge c(n-1)x_n)\le \exp(-f(c)(n-1)x_n),\quad f(c):=1+c(\log c -1).
\]
So, $\pr(R_i\ge 2nx_n)\le e^{-n^{1/2}/3}$ if $n$ is large enough. Invoking \eqref{P_delta}, we have
then
\begin{align*}
\pr\Bigl(R_{\text{max}}\ge 2n^{1/2}\log^{1+\delta}n\Bigr)&\le P_{\delta}+\sum_{i\in [n]}\pr\Bigl(R_i\ge
2n^{1/2}\log^{1+\delta}n\Bigr)\\
&\le e^{-\Theta(\log^{1+\delta} n)}+n e^{-n^{1/2}/3}.
\end{align*}
Thus
\begin{Lemma}\label{R_max<} For $\delta>0$ arbitrarily small, quite surely $R_{\text{max}}$ is of order $n^{1/2}\log^{1+\delta}n$.
\end{Lemma}

Combining Lemma \ref{R_max<} with \eqref{EOn<} in Theorem \ref{E[On]}, we have proved
\begin{Corollary}\label{ABM=>} With high probability, there exists a perfect matching which is blocked by at most $n^{3/4} (\log n)^{2+\delta}$ unmatched pairs.
\end{Corollary}

%{\bf Note} We know (Corollary \ref{fruit}) that q.s. $|\text{Odd}\,(\bold\Pi)|\le n^{1/2}\log n$. So, by
%Lemma \ref{R_max<}, 
%$\mathcal R_{\text{odd}}$, defined as the {\it total\/} rank of the predecessors in the odd parties,
%is q.s. of order $n (\log n)^{2+\delta}$.  
\subsection{Likely range of $\mathcal R(\bold\Pi)$ in a stable, fixed-point free, partition $\bold\Pi$} In Lemma \ref{P_k(P)=} we proved that $\pr_k(\bold\Pi)$ the probability that $\bold\Pi$ is stable and the total
rank of all $n$ predecessors $\mathcal R(\bold\Pi)$ equals $k$, necessarily exceeding $n+|\text{Odd}\,(\bold\Pi)$, is given by
\begin{equation}\label{P_k(P)=,again}
\begin{aligned}
&\pr_k(\bold\Pi)=\idotsint\limits_{\bold x\in [0,1]^{n}}\left[z^{\bar k}\right] F(\bold x,z)\,d\bold x,\\
F(\bold x,z)&:=\prod_{h\in \text{Odd}\,(\bold\Pi)}\!\!\!\!\!\!x_h\,\,\cdot\prod_{(i,j)\notin D(\bold\Pi)}
\!\!\bigl(\bar x_i\bar x_j+zx_i\bar x_j+z\bar x_ix_j\bigr),
\end{aligned}
\end{equation}
where $m:=|\text{Odd}\,(\bold\Pi)|$ and $\bar k:=k-(n+m)$. 
\begin{Theorem}\label{P(R(P)sim)} For $\eps\in (0,1)$, 
\[
\pr\left(\max_{\bold\Pi}\left|\frac{\mathcal R(\bold\Pi)}{n^{3/2}}-1\right|\ge \eps\right) \le e^{-
\Theta(\log^2 n)}.
\]
\end{Theorem}
\begin{proof} Predictably, we will prove the claim via the union bound, i.e. summing the bounds 
of the respective probabilities for the individual partitions. It suffices then to consider the partitions
$\bold\Pi$ with $m\le m_n=\lceil n^{1/2}\log n\rceil$.

First of all, since $F(\bold x,z)$ in \eqref{P_k(P)=,again} is a polynomial of $z$ with non-negative coefficients, we have a Chernoff-type bound: for $k:=\lceil (1+\eps)n^{3/2}\rceil$,
\[
\pr(\mathcal R(\bold\Pi)\ge k)\le I(\bold\Pi,k):=\idotsint\limits_{\bold x\in [0,1]^{n}}\inf_{z\ge 1}\Bigl[z^{-\bar k} F(\bold x,z)\Bigr]\,d\bold x.
\]
The integrand is, at most,
\[
F(\bold x,1)=F(\bold x)\le_b e^{-\frac{s^2}{2}}\!\!\!\!\! \prod_{h\in \text{Odd}\,(\bold\Pi)}\!\!\!\!\!\!x_h,
\]
($s=\sum_{i\in [n]} x_i$), see Lemma \ref{simple1,2}. Therefore the proof of Lemma \ref{summary} delivers, with only notational modification, that
\begin{equation}\label{I-I4}
I(\bold\Pi,k)-I_4(\bold\Pi,k)\le \frac{e^{-\Theta(\log^2 n)}}{(n+m-1)!!}.
\end{equation}
Here $I_4(\bold\Pi,k)$ is the integral of $\inf_{z\ge 1}\Bigl[z^{-\bar k} F(\bold x,z)\Bigr]$ over
$C_4\subset [0,1]^n$, defined by the additional constraints: with 
%$s\le s_n=n^{1/2}+3\log n$, \eqref{s<,max x_i<} and \eqref{sumu_iu_{i+n1/2}/(sumu_j)^2}, where 
$\xi=\sum_{i\in [n_1]}x_i$, ($n_1:=n-m)$,
\begin{align}
&\qquad\qquad s\le s_n:=n^{1/2}+3\log n,\quad \max_{i\in [n]} x_i\le 1.02\, \frac{s\log^2 n}{n};\label{I}\\
%&\Bigl|\frac{n_1}{2}t_1(\bold v)-1\Bigr|\le n^{-\sigma}, \,\,
&\left|\frac{n_1\sum_{i\in [n_1]}x_i^2}{2\xi^2}-1\right|\le n^{-\sigma}, \,
\,\left|\frac{2n_1\sum_{i\in [n_1/2]}x_i x_{i+n_1/2}}{\xi^2}-1\right|\le n^{-\sigma}.\label{II}
\end{align}
Instead of looking for the best $z=z(\bold x)\ge 1$ where $z^{-\bar k} F(\bold x,z)$ attains, or is close
to, its infimum, we confine ourselves to a sub-optimal $z=z(s)\ge 1$ (i.e. dependent on $s$ only), which makes $z^{-\bar k} F(\bold x,z)$ suitably small for all $\bold x\in C_4$. Consider $z\le \frac{2\bar k}{sn}$;
as we shall see shortly, the minimum point of an auxiliary bound for the integrand does satisfy this constraint.

Using $1+x\le e^x$,  the constraints \eqref{I}, \eqref{II} and $z\le \frac{2\bar k}{sn}$, we have
\begin{align*}
&\prod_{(i,j)\notin D(\bold\Pi)}\!\!\!\!\bigl(\bar x_i\bar x_j+zx_i\bar x_j+z\bar x_ix_j\bigr)=\!\!\!\!
\prod_{(i,j)\notin D(\bold\Pi)}\!\!\!\!\bigl(1+(1-2z)x_ix_j+(z-1)(x_i+x_j)\bigr)\\
&\qquad\qquad\le\exp\Bigl(\sum_{(i,j)\notin D(\bold\Pi)}\bigl[(1-2z)x_ix_j +(z-1)(x_i+x_j)\bigr]\Bigr)\\
&\qquad\qquad\qquad\qquad\le_b\exp\Bigl((1-2z)\frac{s^2}{2}+n(z-1) s\Bigr);
\end{align*}
therefore 
\[
z^{-\bar k} F(\bold x,z)\le_b \exp\Bigl((1-2z)\frac{s^2}{2}+n(z-1) s-\bar k\log z\Bigr)\! \prod_{h\in \text{Odd}\,(\bold\Pi)}\!\!\!\!\!\!x_h.
\]
So, applying the identity \eqref{int,prod},
\begin{align}
&\qquad\quad I_4(\bold\Pi,k)\le_b \frac{1}{(n+m-1)!}\int_0^{s_n}\exp(H(z,s))\,ds,\label{I4<}\\
&H(z,s):=(1-2z)\frac{s^2}{2}+n(z-1) s-\bar k\log z+(n+m-1)\log s.\notag
\end{align}
Let us use \eqref{I4<} to prove that 
\begin{equation}\label{I4<expl}
I_4(\bold\Pi,k) \le\frac{e^{-\Theta(\eps^2 n)}}{(n+m-1)!!}.
\end{equation}

As a function of $z$, $H(z,s)$ is convex, and has its absolute minimum at
\[
\bar z=\bar z(s):=\frac{\bar k}{(n-s)s}\sim \frac{\bar k}{ns}< \frac{2\bar k}{ns}.
\]
This {\it decreasing\/} function of $s$ is ``Chernoff-admissible'' when $s$ is such that $z(s)\ge 1$. Let $s_1$ be the smaller root of the (quadratic) equation $\bar z(s)=1$: 
\[
s_1>\frac{\bar k}{n}, \quad s_1=\frac{\bar k}{n}+O(1)=(1+\eps)n^{1/2}+O(1).
\]
Thus our best hope is a function
\[
z(s)=\left\{\begin{aligned}
&\bar z(s),&&\text{if }s\le s_1,\\
&1,&&\text{if }s>s_1.\end{aligned}\right.
\]
{\bf (i)\/} $s>s_1$. Here $s_1\sim\frac{\bar k}{n} > (n+m-1)^{1/2}$, the maximum point of 
\[
h(s):=H(1,s)=-\frac{s^2}{2} +(n+m-1)\log s.
\]
So, arguing as in the proof of Lemma \ref{C1},
\begin{equation}\label{arg}
\begin{aligned}
&\frac{1}{(n+m-1)!}\int_{s_1}^{s_n}\exp(H(1,s))\,ds \le_b \frac{e^{h(s_1)}}{(-h'(s_1))(n+m-1)!}\\
&\le\frac{e^{-\Theta(\eps^2 n)}(n+m-2)!!}{(n+m-1)!}=\frac{e^{-\Theta(\eps^2 n)}}{(n+m-1)!!}.
\end{aligned}
\end{equation}
\noindent {\bf (ii)\/} $s<s_1$. Let $\bar h(s):=H(z(s),s)$. Since $H_z(z(s),s)=0$, we have
\begin{align*}
\bar h'(s)&=\left.H_s(z,s)\right|_{z=z(s)}=\Bigl(s-\frac{\bar k}{n-s}\Bigr)+\Bigl(\frac{\bar k}{s}-n\Bigr)+\frac{n+m-1}{s},\\
\bar h''(s)&=1-\frac{\bar k}{(n-s)^2}-\frac{k-1}{s^2}.
\end{align*}
By the second formula, we have $\bar h''(s)<0$ for $s\le s_n$, i.e. $\bar h(s)$ is concave. By the first
formula, we have 
\begin{align*}
\bar h'\Bigl(\tfrac{\bar k}{n}\Bigr)&\ge -(1+o(1))\frac{\bar k^2}{n^3}+\frac{n^2}{\bar k}\sim \frac{n^{1/2}}{1+\eps}\to\infty,\\
\bar h'(s_1)&=\frac{(n+m-1)-s_1^2}{s_1}\sim -\frac{(2\eps+\eps^2)n^{1/2}}{1+\eps}\to-\infty.
\end{align*}
Thus $\max\,\{\bar h(s): s\le s_n\}$ is attained at a unique point $s_2\in [\bar k/n, s_1]$; in particular,
$s_1-s_2=O(1)$. Since $|\bar h''(s)|=O(n^{1/2})$, it follows\,--\,via Taylor's approximation of $\bar h(s_1)
(=h(s_1))$ at $s_2$\,--\,that $\bar h(s_2)=h(s_1)+O(n^{1/2})$. Therefore, similarly to \eqref{arg}, we
obtain
\[
\frac{1}{(n+m-1)!}\int_{0}^{s_1}\exp(H(z(s),s))\,ds \le \frac{e^{-\Theta(\eps^2 n)}}{(n+m-1)!!}.
\]
This bound together with \eqref{arg} imply \eqref{I4<expl}, which in combination with \eqref{I-I4}
deliver
\[
\pr(\mathcal R(\bold\Pi)\ge k)\le \frac{e^{-\Theta(\log^2 n})}{(n+m-1)!!}.
\]
As in the proof of Corollary \ref{fruit}, it follows that
\[
\pr(\exists\,\bold\Pi: \mathcal R(\bold\Pi)\ge (1+\eps)n^{3/2})\le e^{-\Theta(\log^2 n)}.
\]
Similarly
\[
\pr(\exists\,\bold\Pi: \mathcal R(\bold\Pi)\le (1-\eps)n^{3/2})\le e^{-\Theta(\log^2 n)}.
\]
\end{proof}

\section{$\ex\bigl[\mathcal S_n^2\bigr]$ and the expected number of members with multiple stable
predecessors}
%\label{pr(Pi1,Pi2),ex[O^2]}
First of all
\[
\ex\bigl[(\mathcal S_n)_2\bigr]=\sum_{\bold\Pi_1\neq \bold\Pi_2} \!\!\!\pr(\bold\Pi_1,\bold\Pi_2),
\]
where $\pr(\bold\Pi_1,\bold\Pi_2)$ is the probability that $\bold\Pi_1$ and $\bold\Pi_2$ are both
stable.
%and $m:=|\text{Odd}_{1,2}|\in [m_1,m_2]$. $\bigl(\text{Odd}_{1,2}:=\text{Odd}\,(\bold\Pi_t),\,t=1,2\bigr)$.
By Lemma \ref{p(P1,P2stab)=},
\begin{equation}\label{P_1,2=intF}
\begin{aligned}
&\quad \pr(\bold\Pi_1,\bold\Pi_2)=\idotsint\limits_{\bold x,\bold y\in [0,1]^n}F(\bold x,\bold y)\,d\bold x d\bold y,\\
F(\bold x,\bold y)&:=\prod_{h} x_h\,\cdot\prod_{(i\neq j)}\!
[1-x_ix_j-y_iy_j+(x_i\wedge y_i)(x_j\wedge y_j)],\\
&\quad d\bold x=\prod_{i\in [n]} dx_i,\quad d\bold y=\prod_{i\in [n]:\bold\Pi_1(i)\neq \bold\Pi_2(i)}
\!\!\!\!\!\!\!\!\!\!\!dy_i,
\end{aligned}
\end{equation}
where $h\in \text{Odd}\,(\bold\Pi_{1,2})$, $(i\neq j)\in D_1^c\cap D^c_2$, $(D_t=D(\bold\Pi_t))$,
$y_i=x_i$ if $\bold\Pi_1(i)=\bold\Pi_2(i)$, and for every circuit $\{i_1,\dots,i_{\ell}\}$, ($\ell\ge 4$), formed by alternating pairs matched either in 
$\bold\Pi_1$ or $\bold\Pi_2$, we have : 
\begin{equation}\label{alter}
\begin{aligned}
&\text{either } \,\,x_{i_1}>y_{i_1},\,x_{i_2}<y_{i_2},\dots, x_{i_{\ell}}<y_{i_{\ell}},\\
&\text{or}\qquad\,\, x_{i_1}<y_{i_1},\,x_{i_2}>y_{i_2},\dots, x_{i_{\ell}}>y_{i_{\ell}}.
\end{aligned}
\end{equation}
Let $\mu\!=\!\mu(\bold\Pi_1,\bold\Pi_2)$ be the total number of these circuits, and $2\nu\!=\!2\nu(\bold\Pi_1,$$\bold\Pi_2)$, be their total length. Obviously, there are $2^{\mu}$ ways to select one of
two ``alternation'' sequences described in \eqref{alter} for each of the $\mu$ circuits. Whatever the choice, there are exactly $\nu$ vertices $i$, on those circuits, where $y_i>x_i$ and $\nu$ vertices where $y_i<x_i$. Let
$A$ and $B$ denote the correspondent subsets, $|A|=|B|=\nu$. So
\begin{equation}\label{cone}
\begin{aligned}
&y_i>x_i,\,\, \text{ if }i\in A;\quad y_i<x_i,\,\,\text{ if }i\in B,\\
&y_i=x_i,\,\, \text{ if } i\in [n]\setminus (A\cup B)=\text{Odd}_{1,2}\cup\bigl(\text{Even}_{1,2}\setminus(A\cup B)\bigr),
\end{aligned}
\end{equation}
$\text{Even}_{1,2}:=[n]\setminus \text{Odd}_{1,2}$. 
The individual contributions of these $2^{\mu}$ choices of the inequalities along the circuits to
the integral in \eqref{P_1,2=intF} are all the same. This means that $\pr(\bold\Pi_1,\bold\Pi_2)$ 
equals the RHS integral in \eqref{P_1,2=intF}, with inequalities \eqref{cone} instead of \eqref{alter}, {\it times\/} $2^{\mu}$.

As in the previous section,  we need first to identify the subrange of $(\bold x,\bold y)$ that
provides an asymptotically dominant contribution to the integral, and second to find a sharp approximation for that contribution. 
Like Theorem \ref{E[On]}, the key instrument is the bound for the double-indexed product in the definition of $F(\bold x,\bold y)$ proved in Lemma \ref{simple1,2}:
\begin{equation}\label{F(x,y)<}
\begin{aligned}
&\quad F(\bold x,\bold y)\le_b 
\exp\left(-\frac{s_1^2}{2}-\frac{s_2^2}{2}+\frac{s_{1,2}^2}{2}\right)\,\prod_{h} x_h ,\\
&s_1:=\sum_{i\in [n]}x_i,\,\,s_2:=\sum_{i\in [n]} y_i,\,\,s_{1,2}:=\sum_{i\in [n]}(x_i\wedge y_i).
\end{aligned}
\end{equation}
Here $(\bold x,\bold y)$ are subject to the constraints \eqref{cone}. To make use of this bound, we change the variables of integration:
\begin{equation}\label{change}
x'_i=\left\{\begin{aligned}
&x_i-y_i,&&i\in B,\\
&x_i,&&i\not\in B,\end{aligned}\right.
\quad
y_i'=\left\{\begin{aligned}
&y_i-x_i,&&i\in A,\\
&y_i,&&i\notin A.\end{aligned}\right.
\end{equation}
Here $\bold x',\bold y'\in [0,1]^n$, such that $x_i'=y_i'$ if $\bold\Pi_1(i)=\bold\Pi_2(i)$, and the Jacobian $\partial (\bold x,\bold y)/\partial(\bold x',\bold y')$ equals $1$. Furthermore, switching to $(\bold x',\bold y')$ and introducing
\begin{equation}\label{xij=}
\xi_1=\sum_{i\in [n]\setminus B}x_i'+\sum_{i\in B}y_i',\quad \xi_2=\sum_{i\in B}x_i',\quad \xi_3=\sum_{i\in A}y_i',
\end{equation}
we obtain
\begin{equation}\label{sumsqrs}
\begin{aligned}
&\,\,\,\,\,\,-\frac{s_1^2}{2}-\frac{s_2^2}{2}+\frac{s_{1,2}^2}{2}=-\frac{1}{2}\Biggl(\sum_{i\in [n]\setminus B}x_i'+\sum_{i\in B}y_i'+\sum_{i\in B}x_i'\Biggr)^2\\
&\,\,\,\,\,\,-\frac{1}{2}\Biggl(\sum_{i\in [n]\setminus B}x_i'+\sum_{i\in B}y_i'+\sum_{i\in A}y_i'\Biggr)^2
+\frac{1}{2}\Biggl(\sum_{i\in [n]\setminus B}x_i'+\sum_{i\in B}y_i'\Biggr)^2\\
&=-\frac{1}{2}(\xi_1+\xi_2)^2-\frac{1}{2}(\xi_1+\xi_3)^2+\frac{1}{2}\xi_1^2
=-\frac{1}{2}(\xi_1+\xi_2+\xi_3)^2+\xi_2\xi_3.
\end{aligned}
\end{equation}
Notice that
\begin{equation}\label{sumxi=sumlor}
\begin{aligned}
&\xi_1+\xi_2+\xi_3=\sum_{i\in [n]}x_i'+\sum_{i\in A\cup B}y_i'=\sum_{i\in [n]}(x_i\lor y_i),\\
&\qquad\xi_2+\xi_3=\sum_{i\in B}x_i'+\sum_{i\in A}y_i'=\sum_{i\in [n]}|x_i-y_i|.
\end{aligned}
\end{equation}
In full analogy with the case of $\ex\bigl[\mathcal O_n\bigr]$,  the bound \eqref{F(x,y)<} and the identity \eqref{sumsqrs} will allow us to shrink, in several steps,
 the range of $(\bold x,y)$ to a  core range, on which the integrand $F(\bold x,\bold y)$ can be sharply
 approximated.\\

{\bf (1)\/} Recall that we consider the partitions $\bold\Pi$ with the total length of all odd cycles $m=m(\bold\Pi)\le m_n=[n^{1/2}\log n]$. Our first step is to dispense with the pairs $(\bold\Pi_1,\bold\Pi_2)$ of the partitions such that $2\nu=2\nu(\bold\Pi_1,\bold\Pi_2)\ge 2m_n$. 
%the total length of the attendant even circuits,  is either too large, or too small.  
\begin{Lemma}\label{nu>nu_n} 
%Let $\nu_1=\bigl[\log^2 n\bigr]$, and
%$\nu_2=\bigl[n^{1/2}\log n\bigr]$. Then
\begin{align*}
&\ex\bigl[(\mathcal S_n)_2\bigr]-\ex_1\bigl[(\mathcal S_n)_2\bigr]\le e^{-\Theta(\log^2 n)} ,\\
&\ex_1\bigl[(\mathcal S_n)_2\bigr]:=\sum_{\nu(\bold\Pi_1,\bold\Pi_2)\le m_n}
\!\!\!\!\!\!\!\!\!\pr(\bold\Pi_1,\bold\Pi_2).
\end{align*}
%(2) 
%\[
%\sum_{\nu(\bold\Pi_1,\bold\Pi_2)\ge \nu_n}\!\!\!\!\!\!\!\!\pr(\bold\Pi_1,\bold\Pi_2)\le e^{-\Theta(\log^2 n)}.
%\]
%Equivalently,
%\[
%\ex\bigl[(\mathcal O_n)_2\bigr]-\ex_1\bigl[(\mathcal O_n)_2\bigr]\le e^{-\Theta(\log^2 n)},
%\quad \ex_1\bigl[(\mathcal O_n)_2\bigr]:=\sum_{\nu(\bold\Pi_1,\bold\Pi_2)< \nu_n}\!\!\!\!\!\!\!\!\pr(\bold\Pi_1,\bold\Pi_2).
%\]
\end{Lemma}
%\begin{Remark} Thus, since $\ex\bigl[(\mathcal O_n)_2\bigr]$ is of order $n^{3/2}$ at least, $\ex_1\bigl[(\mathcal O_n)_2\bigr]$ is a sharp approximation of this factorial moment. 
%\end{Remark} 
\begin{proof}
By the equations \eqref{F(x,y)<} and \eqref{sumsqrs}, {\it and\/} the identity \eqref{int,prod}, we have
\begin{equation}\label{sumxi_j}
\begin{aligned}
 &\pr(\bold\Pi_1,\bold\Pi_2)\le_b\,2^{\mu}\idotsint\limits_{\bold x',\bold y'\ge \bold 0}
\exp\Biggl(\!-\frac{1}{2}\Bigl(\sum_j\xi_j\Bigr)^2+\xi_2\xi_3\!\Biggr)\\
&\qquad\qquad\qquad\qquad\qquad\qquad\quad\times \Biggl(\,\prod_{h\in \text{Odd}_{1,2}}\!\!\!\! x_h'\Biggr)\,\,\,\prod_{i\in [n]} dx'_i\prod_{j\in A\cup B} dy'_j\\
&=2^{\mu}\iiint\limits_{\xi_j\ge \bold 0}\!
\exp\Biggl(\!-\frac{1}{2}\Bigl(\sum_j\xi_j\Bigr)^2+\xi_2\xi_3\!\Biggr)
\frac{\xi_1^{n+m-1}}{(n+m-1)!}\cdot\frac{(\xi_2\xi_3)^{\nu-1}}
{[(\nu-1)!]^2}\,d\boldsymbol{\xi}.
\end{aligned}
\end{equation}
Expanding $\exp(\xi_2\xi_3)=\sum_{k\ge 0}\xi_2^k\xi_3^k/k!$ and using again, term-wise, \eqref{int,prod}, we obtain
\begin{equation}\label{explsum}
\begin{aligned}
\pr(\bold\Pi_1,\bold\Pi_2)&\le_b2^{\mu}\sum_{k\ge 0}s(n+m,\nu,k),\\
s(n+m,\nu,k)&:=\frac{[(\nu-1+k)!]^2}{[(\nu-1)!]^2k!(n+m+2(\nu+k)-1)!!}.
\end{aligned}
\end{equation}
For $m=0$ this sum was estimated in \cite{Pit2}. For our case the estimate from \cite{Pit2} becomes
\begin{equation}\label{sum_k}
\sum_{k\ge 0}s(n+m,\nu,k)\le_b n\left(\frac{e}{n+m}\right)^{\frac{n+m}{2}}\!\!
(n+m)^{-\nu}.
\end{equation}
Furthermore, the number of ordered pairs $(\bold\Pi_1,\bold\Pi_2)$ with parameters $m$, $\nu$ and
$\mu$ is
\begin{equation}\label{countpairs}
\binom{n}{m} f(m) \binom{n-m}{2\nu}(n-m-2\nu-1)!!\cdot 2^{\mu}f(2\nu,\mu);
\end{equation}
here, as we recall, $f(m)$ is the total number of permutations of $[m]$ with only odd cycles of length
$3$ or more, and $f(2\nu,\mu)$ is the total number of circuit partitions of $[2\nu]$ with $\mu$  circuits, 
each of even length $4$ at least. The factor $2^{\mu}$ counts the total number of ways to assign,
in the alternating fashion, the edges of the circuits to the matching sets of $\bold\Pi_1$ and $\bold\Pi_2$.
Clearly then, $2^{\mu}f(2\nu,\mu)$ is the total number of permutations of $[2\nu]$ with only even
cycles, of length $4$ at least. We add that $(n-m-2\nu-1)!!$ is the total number of ways to form the
$(n-m-2\nu)/2$ matched pairs out of $n-m-2\nu$ elements outside the circuits, i.e. the pairs {\it common\/} to $\bold\Pi_1$ and $\bold\Pi_2$.

So, by \eqref{explsum}, \eqref{sum_k} and \eqref{countpairs}, we obtain
\begin{equation}\label{clumsy}
\begin{aligned}
&\sum_{\nu(\bold\Pi_1,\bold\Pi_2)\ge m_n}\!\!\!\!\!\!\!\!P(\bold\Pi_1,\bold\Pi_2)\le_b n\!\!\!\!\!\sum_{m\le
m_n}\!\binom{n}{m} f(m)
\left(\!\frac{e}{n+m}\!\right)^{\frac{n+m}{2}}\\
&\qquad\times\sum_{\nu\ge m_n}\!\!\binom{n-m}{2\nu}(n-m-2\nu-1)!! (n+m)^{-\nu}\sum_{\mu}2^{2\mu}f(2\nu,\mu).
\end{aligned}
\end{equation}
Now $f(m)\le m!$, and from \cite{Pit2} (Appendix) it follows that
\begin{equation}\label{sum_{mu}}
\sum_{\mu}2^{2\mu}f(2\nu,\mu)=e^{-1+O(\nu^{-1})} (2\nu)!=O((2\nu)!).
\end{equation}
Also
\[
\frac{(n-m-2\nu-1)!!}{(n-m-2\nu)!}=\left[2^{\frac{n-m-2\nu}{2}}\left(\frac{n-m-2\nu}{2}
\right)!\right]^{-1}.
\]
So the bound \eqref{clumsy} yields 
\begin{align*}
&\sum_{\nu(\bold\Pi_1,\bold\Pi_2)\ge m_n}\!\!\!\!\!\!\!\!P(\bold\Pi_1,\bold\Pi_2)\le_b
n!\,n\,\cdot\!\!\!\sum_{m\le m_n}\!\!\left(\!\frac{e}{n+m}\!\right)^{\frac{n+m}{2}}\\
&\qquad\qquad\qquad\times\sum_{\nu\ge m_n}\!
\left[(n+m)^{\nu}\, 2^{\frac{n-m-2\nu}{2}}\left(\frac{n-m-2\nu}{2}
\right) ! \right]^{-1}\\
&\le n!\,n^2\,\cdot\!\!\!\!\!\sum_{m\le m_n}\!\!\left(\!\frac{e}{n+m}\!\right)^{\frac{n+m}{2}}
\!\!\left[(n+m)^{m_n}\, 2^{\frac{n-m-2m_n}{2}}\left(\frac{n-m-2m_n}{2}
\right) !\,\right]^{-1}\!\!\!,
\end{align*}
since in the $\nu$-sum the terms decrease with $\nu$.
Applying Stirling formula for the two factorials and using $m\le m_n\ll n^{2/3}$ in the
expansions of $\log (1+z)$, $(z=m/n,\,z=-\tfrac{m+2m_n}{n})$, we transform this bound into
\begin{equation}\label{sum_{nu>nu_2}}
\sum_{\nu(\bold\Pi_1,\bold\Pi_2)\ge m_n}\!\!\!\!\!\!\!\!P(\bold\Pi_1,\bold\Pi_2)
\le n^2\cdot\!\!\! \sum_{m\le m_n}\!\!\!\exp\left(\!-\frac{(m+2m_n)^2}{4n}\right)\le e^{-0.99\log^2 n}.
\end{equation}
\end{proof}
% Likewise, but using 
%\[
%\frac{f(m)}{m!}=O(m^{-1/2}),
%\]
%(see \eqref{f(m)sim}), we have 
%\begin{equation}\label{sum{nu<nu_1}}
%\begin{aligned}
%\sum_{\nu(\bold\Pi_1,\bold\Pi_2)\le\nu_1}\!\!\!\!\!\!\!\!P(\bold\Pi_1,\bold\Pi_2)
%&\le n \nu_1\cdot\!\!\! \sum_{m\in [m_1,m_2]}m^{-1/2}\exp\Bigl(\!-\frac{m^2}{4n}\Bigr)
%\\
%&\le_b n^{5/4} \nu_1=O\bigl(n^{5/4}\log^2 n\bigr).
%\end{aligned}
%\end{equation}
%The equations \eqref{sum_{nu>nu_2}} and \eqref{sum{nu<nu_1}} complete the proof.
%\end{proof}
From now on we will consider only {\it admissible\/} pairs $(\bold\Pi_1,\bold\Pi_2)$, i.e. those
satisfying $m(\bold\Pi_1,\bold\Pi_2)\le m_n$ and $\nu(\bold\Pi_1,\bold\Pi_2)\le m_n$.\\

{\bf (2)\/}  For the admissible pairs $(\bold\Pi_1,\bold\Pi_2)$, we can discard large parts of the $(\bold x,\bold y)$'s range, like we did
for individual partitions $\bold\Pi$ in the case of $\ex\bigl[\mathcal S_n\bigr]$. 

For a generic set $\mathcal C$ of $(\bold x,y)$ with $\bold x,\bold y\in [0,1]^n$, we define
\begin{align*}
&\pr_{\mathcal C}(\bold\Pi_1,\bold\Pi_2)=\idotsint\limits_{(\bold x,\bold y)\in \mathcal C}F(\bold x,\bold y)\,d\bold x d\bold y, 
\quad E_{\mathcal C}\bigl[(\mathcal S_n)_2\bigr]=\sum_{\bold\Pi_1,\bold\Pi_2}\!\!\!\pr_{\mathcal C}(\bold\Pi_1,\bold\Pi_2).
\end{align*}
\begin{Lemma}\label{sum(lor)} Introducing $s_n=n^{1/2}+6\log n$, and 
$
\mathcal C_1=\Bigl\{\bold x,\bold y: \sum_{i\in [n]}(x_i\lor y_i)\le s_n\Bigr\},
%&E_2\bigl[(\mathcal O_n)_2\bigr]=\sum_{\bold\Pi_1,\bold\Pi_2}\!\!\!\pr_{\mathcal C_1}(\bold\Pi_1,\bold\Pi_2),
$
we have
\[
E_1\bigl[(\mathcal S_n)_2\bigr]- E_{\mathcal C_1}\bigl[(\mathcal S_n)_2\bigr]\le e^{-\Theta(\log^2n)}.
\]
\end{Lemma}
\begin{proof} We already observed, \eqref{sumxi=sumlor}, that $\sum_i (x_i\lor y_i)=\sum_j\xi_j$. 
So,
%by the definition of $\pr(\bold\Pi_1,\bold\Pi_2)$, $\pr_{\mathcal C_1}(\bold\Pi_1,\bold\Pi_2)$, and the
%inequality \eqref{F(x,y)<}, 
similarly to \eqref{sumxi_j}-\eqref{explsum} we have: 
%with $\nu:=\nu(\bold\Pi_1,
%\bold\Pi_2)\in [\nu_1,\nu_2]$, $m:=m(\bold\Pi_1,\bold\Pi_2)\in [m_1,m_2]$, and $\mu=\mu(\bold\Pi_1,\bold\Pi_2)$ being the total number of  circuits, of even size $\ge 4$, formed by the two partitions,
\begin{multline}\label{P-P_1}
\pr(\bold\Pi_1,\bold\Pi_2)-\pr_{\mathcal C_1}(\bold\Pi_1,\bold\Pi_2)\\
\le_b 2^{\mu}\!\!\!\!\iiint\limits_{\xi_1+\xi_2+\xi_3\ge s_n}\!
\!\!\!\!\exp\Biggl(\!-\frac{1}{2}\Bigl(\sum_j\xi_j\Bigr)^2+\xi_2\xi_3\!\Biggr)
\frac{\xi_1^{n+m-1}}{(n+m-1)!}\cdot\frac{(\xi_2\xi_3)^{\nu-1}}
{[(\nu-1)!]^2}\,d\boldsymbol{\xi}\\
= 2^{\mu}\sum_{k\ge 0}\frac{[(\nu-1+k)!]^2}{[(\nu-1)!]^2k!(n+m+2(\nu+k)-1)!}\\
\times\int_{s\ge s_n}\!\!\exp\Bigl(-\frac{s^2}{2}\Bigr) s^{n+m+2(\nu+k)-1}\,ds.
\end{multline}
(Relaxing the constraint on $s$ to $s\ge 0$ we get back to \eqref{explsum}.) The last integrand attains its maximum at
\[
s_{\text{max}}=\bigl(n+m+2(\nu+k)-1\bigr)^{1/2},
\]
which is below $s_n-3\log n$ if $k\le m_n$. Let $S_{\le m_n}$ and $S_{>m_n}$ denote the
sub-sums of the sum above, for $k\le m_n$ and $k>m_n$ respectively. Then, expanding integration to 
$[0,\infty)$, we obtain
\begin{align*}
S_{>m_n}&\le \sum_{k>m_n}\frac{[(\nu-1+k)!]^2}{[(\nu-1)!]^2\,k!\,(n+m+2(\nu+k)-1)!!}\\
&\le_b \frac{[(\nu+m_n)!]^2}{[(\nu-1)!]^2\,m_n!\,(n+m+2(\nu+m_n)+1)!!};
\end{align*}
since $\nu\le m_n$, the ratio of the consecutive terms in the sum is below $2/3$. Droping
$[(\nu-1)!]^2$ in the denominator and using the Stirling formula for the other factorials, we simplify the
bound to
\begin{equation}\label{S_><}
S_{>m_n}\le_b \left(\frac{e}{n+m}\right)^{\frac{n+m}{2}}\!\! (n+m)^{-(\nu+m_n)}.
\end{equation}
The bound is smaller than the bound \eqref{sum_k} for the full sum of $s(n+m,\nu,k)$ by the
factor $(n+m)^{m_n}$. Turn to $S_{\le m_n}$. This time the bottom integral over $s\ge s_n$ in \eqref{P-P_1} is small,
compared to the integral over all $s\ge 0$, because for $k\le m_n$ the maximum point of the
integrand is at distance $3\log n$, at least, from the interval $[s_n,\infty)$. More precisely, using the
argument in the proof of Lemma \ref{C1}, we have
\[
\int_{s\ge s_n}\!\!\exp\Bigl(-\frac{s^2}{2}\Bigr) s^{n+m+2(\nu+k)-1}\,ds\le_b e^{-8\log^2 n}\bigl(n+m +2(\nu+k)-2\bigr)!!.
\]
Therefore
\begin{equation}\label{S<<}
\begin{aligned}
S_{\le m_n}&\le_b e^{-8\log^2 n}\sum_{k\le m_n}\frac{[(\nu-1+k)!]^2}{[(\nu-1)!]^2k!(n+m+2(\nu+k)-1)!!}\\
&\le e^{-8\log^2 n}\sum_{k\ge 0}\frac{[(\nu-1+k)!]^2}{[(\nu-1)!]^2k!(n+m+2(\nu+k)-1)!!}\\
&=e^{-8\log^2 n}\sum_{k\ge 0}s(n+m,\nu,k).
\end{aligned}
\end{equation}
Combining \eqref{S_><}, \eqref{S<<} and \eqref{sum_k} we transform the inequality \eqref{P-P_1}
into
\begin{equation}\label{P-P1expl}
P(\bold\Pi_1,\bold\Pi_2)-P_{\mathcal C_1}(\bold\Pi_1,\bold\Pi_2)\le e^{-\Theta(\log^2 n)}\, 2^\mu \!\left(\frac{e}{n+m}\!\right)^{\frac{n+m}{2}}\!\!
(n+m)^{-\nu}.
\end{equation}
So,  like the part {\bf (1)\/} in  the  proof of Lemma \ref{nu>nu_n},
\begin{equation}\label{sum(P(P1,P_2)-P_1(P_1,P_2))}
\begin{aligned}
&\quad\sum_{\bold\Pi_1,\bold\Pi_2}\bigl[P(\bold\Pi_1,\bold\Pi_2)-P_{\mathcal C_1}(\bold\Pi_1,\bold\Pi_2)]
\le e^{-\Theta(\log^2n)}n!\,m_n\\
&\times\!\!\! \sum_{m\le m_n]}\!\!\!\left(\frac{e}{n+m}\right)^{\frac{n+m}{2}}
\left[(n+m)^{0}\, 2^{\frac{n-m-2\cdot 0}{2}}\left(\frac{n-m-2\cdot 0}{2}
\right) ! \right]^{-1}\\
&\qquad\qquad\qquad=e^{-\Theta(\log^2n)}.
\end{aligned}
\end{equation}
\end{proof}

We need some additional reduction of the last
range $\mathcal C_2$. The bound \eqref{F(x,y)<} will continue to be the key tool, 
until the resulting range is narrow enough to permit a sufficiently sharp bound of the double product
\[
G(\bold x,\bold y)=\prod_{(i\neq j)\in D^c_1\cap D^c_2}\!\!\!\!\!\!\!\bigl[1-x_ix_j-y_iy_j+(x_i\wedge y_i)(x_j\wedge y_j)\bigr]
\]
in \eqref{P_1,2=intF}.  %To do so, we need a sharper bound for $G(x,y)$. To see what needs to be done, %For guidance, let us  turn back to the proof of Lemma \ref{simple1,2}. 
%Define $\mathcal N=\mathcal N(\bold\Pi_1,
%\bold\Pi_2)$ as the vertex set of $M_1\cap M_2$. 
Define $\mathcal N=\mathcal N(\bold\Pi_1,\bold\Pi_2)$ and $\mathcal M=\mathcal M(\bold\Pi_1,\bold\Pi_2)$ as the vertex set of all odd
cycles and even cycles, of length $4$ or more,  and the vertex set of the edges common to
both partitions, respectively. So $|\mathcal N|=m+2\nu$, and $|\mathcal M|=n-(m+2\nu)$. Arguing as
in the proof of Lemma \ref{simple1,2},  but retaining more
terms,  we have 
%By \eqref{sumuiuj}, we have
%\begin{align*}
%&\sum_{(i\neq j)\in D^c_1\cap D^c_2}\!\!\!\!\!\!\!\bigl[x_ix_j+y_iy_j-(x_i\wedge y_i)(x_j\wedge y_j)\bigr]\\
%&=\frac{s_1^2}{2}+\frac{s_2^2}{2}-\frac{s_{1,2}^2}{2}-\frac{1}{2}\sum_{i\in \mathcal M}x_i^2
%-\sum_{(i\neq j)\in M_1\cap M_2}\!\!\!\!\!\!x_ix_j\,\,
%+O\Bigl(\sum_{i\in \mathcal N}
%(x_i^2+y_i^2)\Bigr).
%$\end{align*}
%Further, denoting $z_i=x_i\wedge y_i$ and using \eqref{sumuiuj} again,
%\begin{align*}
%\sum_{(i\neq j)\in D^c_1\cap D^c_2}&\!\!\!\!\!\!\!\bigl[x_ix_j+y_iy_j-(x_i\wedge y_i)(x_j\wedge y_j)\bigr]^2
%\ge \sum_{(i\neq j)\in D^c_1\cap D^c_2}\!\!\!z_i^2z_j^2\\
%&=\frac{1}{2}\Bigl(\sum_{i\in [n]} z_i^2\Bigr)^2 +O\Bigl(\sum_{i\in [n]}z_i^4\Bigr).
%\end{align*}
%Consequently
\begin{equation}\label{G(x,y)<expl}
\begin{aligned}
G(\bold x,\bold y)&\le \exp\Biggl(\!-\frac{s_1^2}{2}-\frac{s_2^2}{2}+\frac{s_{1,2}^2}{2}+
\frac{1}{2}\sum_{i\in \mathcal M}x_i^2 +\sum_{(i\neq j)\in M_1\cap M_2}\!\!\!\!\!\!x_ix_j\\
%\end{aligned}
%\end{equation}
%\begin{equation*}
%\begin{aligned}
&\,\,\,\,-\frac{1}{4}\Bigl(\sum_{i\in [n]} (x_i\wedge y_i)^2\Bigr)^2 +O\Bigl(\sum_{i\in \mathcal N}(x_i^2+y_i^2)\!\Bigr)+O\Bigl(\sum_{i\in [n]}x_i^4\Bigr)\!\Biggr).
\end{aligned}
\end{equation}
Thus we  have to find sharp
approximations of the three explicit sums and to establish the $o(1)$ bounds of the remainders for almost all $(\bold x,\bold y)\in \mathcal C_1$. With those approximations at hand we will obtain an explicit upper bound for $\ex\bigl[(\mathcal S_n)_2\bigr]$.  For brevity will not present a proof of a matching lower bound. \\

{\bf (3)\/} By \eqref{change} and \eqref{xij=}, $s:=\xi_1+\xi_2+\xi_3=\sum_{i\in [n]}x_i'+\sum_{i\in A\cup B}y_i'.$
\begin{Lemma}\label{C1-C2} Define $\bold u'=\{u_i'\}_{i\in [n]}$, where $u_i'=x_i'/s$, for $i\in [n]$, and $u_i'=y_i'/s$ for $i\in A\cup B$. Define $T_1(\bold u')=\max_i u_i'$. 
For
\[
\mathcal C_2:=\Bigl\{(\bold x,\bold y)\in \mathcal C_1: T_1(\bold u')\le 1.01\frac{\log^2 n}{n}\Bigr\},
\]
we have
\begin{equation*}
\pr_{\mathcal C_1}(\bold\Pi_1,\bold\Pi_2)-\pr_{\mathcal C_2}(\bold\Pi_1,\bold\Pi_2)
\le 2^{\mu} e^{-\Theta(\log^2 n)}\left(\frac{e}{n+m}\right)^{\frac{n+m}{2}}\!\!n^{-\nu}.
\end{equation*}
\end{Lemma}
\begin{proof}
Introduce $L_1',\dots,L_{n+2\nu}'$, the intervals lengths in the random partition of $[0,1]$ by the
$n+2\nu-1$ random points. Analogously to \eqref{sumxi_j}, but using the sharper inequality in Lemma \ref{intervals1}, \eqref{joint<}, we have: with $s:=\xi_1+\xi_2+\xi_3$,
\begin{equation*}%\label{max<}
\begin{aligned}
 &\pr_{\mathcal C_1}(\bold\Pi_1,\bold\Pi_2)-\pr_{\mathcal C_2}(\bold\Pi_1,\bold\Pi_2)\\
 &\le_b\,2^{\mu}\!\!\!\!\!\!\idotsint\limits_{\bold x', \bold y'\ge \bold 0\atop T_1(\bold u')>1.01\frac{\log^2 n}{n}}\!\!\!\!\!
e^{-\frac{s^2}{2}+\xi_2\xi_3}
%\exp\Bigl(-\frac{(\xi_1+\xi_2+\xi_3)^2}{2}+\xi_2\xi_3\Bigr)\\
%&\qquad\qquad\qquad\qquad\qquad\qquad\times
\prod_{h\in\text{Odd}_{1,2}}\!\!\!\!x_h\,\prod_{i\in [n]} dx'_i\prod_{j\in A\cup B}\!\! dy'_j\\
%&\qquad\qquad\qquad\qquad\qquad\qquad\quad\times \Biggl(\,\prod_{h\in \text{Odd}_{1,2}}\!\!\!\! x_h'\Biggr)\,\,\,\prod_{i\in [n]} dx'_i\prod_{j\in A\cup B} dy'_j\\
&\le \frac{2^{\mu}}{(n+2\nu-1)!}\,\ex\Biggl[\chi\Bigl(T_1(\bold L')\ge 1.01\frac{\log^2n}{n}\Bigr)\prod_{h\in \text{Odd}_{1,2}}\!\!\!\!\!L_h'\Biggr]\\
&\quad\times \idotsint\limits_{\bold x', \bold y'\ge \bold 0}
e^{-\frac{s^2}{2}+\xi_2\xi_3}\,s^m\prod_{i\in [n]} dx'_i\prod_{j\in A\cup B}\!\! dy'_j.
%\\
%&\iiint\limits_{\xi_j\ge \bold 0}\!
%\exp\Bigl(-\frac{(\xi_1+\xi_2+\xi_3)^2}{2}+\xi_2\xi_3\Bigr)
%\exp\Biggl(\!-\frac{1}{2}\Bigl(\sum_j\xi_j\Bigr)^2+\xi_2\xi_3\!\Biggr)
%\frac{\xi_1^{n+m-1}}{(n+m-1)!}\cdot\frac{(\xi_2\xi_3)^{\nu-1}}
%{[(\nu-1)!]^2}\,d\boldsymbol{\xi}.
\end{aligned}
\end{equation*}
Arguing as in \eqref{union}, the expectation factor is less than
\[
e^{-1.01\log^2 n}\frac{(n+2\nu)!}{(n+m+2\nu-2)!}. 
\]
The integral is less than
\begin{align*}
I_n(m,\nu)&:=\iiint\limits_{\xi_j\ge 0} e^{-\frac{s^2}{2}+\xi_2\xi_3} s^m\frac{\xi_1^{n-1}}{(n-1)!}
\cdot\frac{(\xi_2\xi_3)^{\nu-1}}{\bigl[(\nu-1)!\bigr]^2}\,d\boldsymbol\xi\\
&=\sum_{k\ge 0}\frac{\bigl[(\nu-1+k)!\bigr]^2}{(n-1)!\,\bigl[(\nu-1)!\bigr]^2\,k!\,\bigl(2(\nu+k)-1\bigr)!}\\
&\times\iint\limits_{\xi_1,\xi_4\ge 0}e^{-\frac{(\xi_1+\xi_4)^2}{2}} (\xi_1+\xi_4)^m\, \xi_1^{n-1}\xi_4^{2(\nu+k)-1}\,d\xi_1 d\xi_4.
\end{align*}
Here the double integral equals
\[
\frac{(n-1)!\,\bigl(2(\nu+k)-1\bigr)!\,\bigl(n+m+2(\nu+k)-2\bigr)!!}{\bigl(n+2(\nu+k)-1\bigr)!}.
\]
So
\[
I_n(m,\nu)=\sum_{k\ge 0}\frac{\bigl[(\nu-1+k)!\bigr]^2\,\bigl(n+m+2(\nu+k)-2\bigr)!!}
{\bigl[(\nu-1)!\bigr]^2\,k!\,\bigl(n+2(\nu+k)-1\bigr)!}
\]
Therefore 
\begin{align*}
&\quad\pr_{\mathcal C_1}(\bold\Pi_1,\bold\Pi_2)-\pr_{\mathcal C_2}(\bold\Pi_1,\bold\Pi_2)
\le_b 2^{\mu} e^{-\log^2 n}\sum_{k\ge 0} s'(n,m,\nu,k),\\
&s'(n,m,\nu,k):=\frac{\bigl[(\nu-1+k)!\bigr]^2\bigl(n+m+2(\nu+k)-2\bigr)!!\,(n+2\nu)!}{(n+m+2\nu-2)!\,\bigl[(\nu-1)!\bigr]^2\,k!\,\bigl(n+2(\nu+k)-1\bigr)!}.
\end{align*}
The summand $s'(n,m,\nu,k)$ is similar to the summand $s(n+m,\nu,k)$ defined in \eqref{explsum}.
Closely following the derivation of the bound for $\sum_{k\ge 0}s(n,\nu,k)$ in \cite{Pit2}, we obtain
\begin{equation*}
\sum_{k\ge 0} s'(n,m,\nu,k)\le_b n^2\left(\frac{e}{n+m}\right)^{\frac{n+m}{2}}\!\!n^{-\nu},
\end{equation*}
compare to \eqref{sum_k}. The last two bounds complete the proof.
\end{proof}
On $\mathcal C_1\supset \mathcal C_2$ we have 
\[
s=\sum_{ i\in [n]}x_i'+\sum_{i\in A\cup B}y_i'=\sum_{i\in [n]}(x_i\lor y_i)\le s_n=n^{1/2}+6\log n.
\]
and on $\mathcal C_2$
\[
\max\Bigl\{\max_i\frac{x_i'}{s},\,\max_{j\in A\cup B}\frac{y_j'}{s}\le 1.01\frac{\log^2 n}{n}\Bigr\}.
\]
Since $m,\,\nu\le n^{1/2}\log n$, we have then the counterparts of the bounds in \eqref{sumx_h^2}.
Namely,  on $\mathcal C_2$,
\begin{equation}\label{onC_2}
\begin{aligned}
\sum_{i\in \mathcal N}(x_i'+&y_i')\le_b sn^{-1}\log^3n,\quad \sum_{i\in \mathcal N}(x_i'+y_i')^2\le_b n^{-1/2}\log^5n,\\
&\sum_{i\in \mathcal N}(x_i'+y_i')^4\le_b n^{-1}\log^8 n,
%\quad \sum_{i\in \mathcal N\cup\mathcal M}(x_i'+y_i')^6\le_b n^{-2}\log^{12} n,
\end{aligned}
\end{equation}
\\

%{\bf (4)\/} With $\xi:=\sum_{i\in \mathcal N^c}x_i'+\sum_{j\in A\cup B}y_j'$, define $v_i'=x_i'/\xi$ for
%$i\in \mathcal N^c$ and $v_j'=y_j'/s$ for $j\in A\cup B$. Introduce $T_2(\bold v')=\sum_{i\in \mathcal (v_i')^2+\sum_{j\in (A\cup B)}(v_j')^2$.
%\begin{Lemma}\label{P_{C2}(P_1,P_2)-P_{C3}(P_1,P_2)} For $\sigma<1/3$, let
%\[
%\mathcal C_3=\left\{(\bold x,\bold y)\in \mathcal  C_2:\left|\frac{n-m+2\nu}{2}\,T_2(\bold v)-1\right|\le
%n^{-\sigma}\right\}; 
%\]
%here $(n-m+2\nu=|\mathcal N^c|+|A\cup B|)$. Then
%\begin{equation*}
%\pr_{\mathcal C_2}(\bold\Pi_1,\bold\Pi_2)-\pr_{\mathcal C_3}(\bold\Pi_1,\bold\Pi_2)
%\le_b 2^{\mu} e^{-\Theta(\log^2 n)}\left(\frac{e}{n+m}\right)^{\frac{n+m}{2}}\!\!n^{-\nu}.
%\end{equation*}
%\end{Lemma}

%\begin{proof} This time we need the random partition of $[0,1]$ into the intervals $\mathcal L_1',\dots, 
%\mathcal L_{n-m+2\nu}'$.

{\bf (4)\/} With $\xi:=\sum_{i\in \mathcal M}x_i'(=\sum_{i\in \mathcal M}x_i)$, define $v_i'=x_i'/\xi$ for
$i\in \mathcal M$, $|\mathcal M|=n-m-2\nu$. Introduce $T_2(\bold v')=\sum_{i\in \mathcal M}(v_i')^2$, and $V'$ the set of all $\bold v'$ such that 
\[
\left|\frac{n-m-2\nu}{2}\,T_2(\bold v')-1\right|\le n^{-\sigma}.
\]
\begin{Lemma}\label{C2-C3} For $\sigma<1/3$, let
$
\mathcal C_3=\bigl\{(\bold x,\bold y)\in \mathcal  C_2: \bold v'\in V\bigr\}.
$
Then
\begin{equation*}
\pr_{\mathcal C_2}(\bold\Pi_1,\bold\Pi_2)-\pr_{\mathcal C_3}(\bold\Pi_1,\bold\Pi_2)
\le_b 2^{\mu}\, e^{-\Theta(n^{1/3-\sigma})}\left(\frac{e}{n+m}\right)^{\frac{n+m}{2}}\!\!n^{-\nu}.
\end{equation*}
\end{Lemma}

\begin{proof} Introduce 
\[
\xi_4:=\xi_1-\xi=\sum_{i\in B^c\cap \mathcal M^c}x_i'+\sum_{i\in B}y_i'=\sum_{i\in \text{Odd}_{1,2}}x_i'+
\sum_{i\in A} x_i'+\sum_{i\in B}y_i'.
\]
Then with $s:=\xi+\xi_4+\xi_2+\xi_3$,
\begin{align*}
&\pr_{\mathcal C_2}(\bold\Pi_1,\bold\Pi_2)-\pr_{\mathcal C_3}(\bold\Pi_1,\bold\Pi_2)\\
&\qquad\quad \le_b\,2^{\mu}\idotsint\limits_{\bold x',\,\bold y':\, \bold v'\in V} e^{-\frac{s^2}{2}+\xi_2\xi_3}
\prod_{i\in M} dx'_i\, \prod_{j\in\text{Odd}_{1,2}}\!\!\!\!x_j dx_j\\
&\qquad\qquad\qquad\qquad\quad\times\prod_{k\in A}dx_k'\,\prod_{\ell\in B}dy_{\ell}'\,\prod_{b\in B} dx_b'\,\prod_{a\in A}dy_a'.
%2^{\mu}\!\!\!\!\!\!\!\!\!\!\idotsint\limits_{ \left|\frac{|M|}{2}\,T_2(\bold v')-1\right|\le
%n^{-\sigma}} e^{-\frac{s^2}{2}+\xi_2\xi_3}
\end{align*}
Now the integrand depends on $\{x_i'\}_{i\in \mathcal M}$ only through $\xi=\sum_{i\in \mathcal M}x_i'$. So, introducing 
the random intervals $\mathcal L_1',\dots, \mathcal L_{n-m-2\nu}'$ forming the partition of $[0,1]$, we obtain
\begin{align*}
&\pr_{\mathcal C_2}(\bold\Pi_1,\bold\Pi_2)-\pr_{\mathcal C_3}(\bold\Pi_1,\bold\Pi_2)\\
&\le_b 2^{\mu}\!\pr\Bigl(\Bigl|\frac{n-m-2\nu}{2}T_2(\boldsymbol{\mathcal L}')-1\Bigr|\ge n^{-\sigma}\!\Bigr)
\iiiint\limits_{\xi, \,\xi_j\ge 0}e^{-\frac{s^2}{2}+\xi_2\xi_3}\frac{\xi^{|\mathcal M|-1}\,d\xi}{(|\mathcal M|-1)!}\\
&\qquad\qquad\qquad\times \frac{\xi_4^{2m+2\nu-1}\,d\xi_4}{(2m+2\nu-1)!}\cdot
 \frac{\xi_2^{\nu-1}\,d\xi_2}{(\nu-1)!}\cdot\frac{\xi_3^{\nu-1}\,d\xi_3}{(\nu-1)!}.
\end{align*}
The probability is of order $e^{-\Theta(n^{1/3-\sigma})}$, and the integral equals the bottom
$3$-dimensional integral in \eqref{sumxi_j}. Jointly with \eqref{explsum} and \eqref{sum_k}
this proves the claim.
\end{proof}

Finally, introduce $T_3(\bold v')=\sum\limits_{(i,j)\in M_1\cap M_2}\!\!v_i' v_j'$;  (here, of course, $i,j\in\mathcal M$).
\begin{Lemma}\label{C3-C4} For $\sigma<1/3$, let
\[
\mathcal C_4=\left\{(\bold x,\bold y)\in \mathcal  C_3:\left|2(n-m-2\nu)T_3(\bold v')-1\right|\le
n^{-\sigma}\right\}; 
\]
Then
\begin{equation*}
\pr_{\mathcal C_3}(\bold\Pi_1,\bold\Pi_2)-\pr_{\mathcal C_4}(\bold\Pi_1,\bold\Pi_2)
\le 2^{\mu} e^{-\Theta(n^{1/3-\sigma})}\left(\frac{e}{n+m}\right)^{\frac{n+m}{2}}\!\!n^{-\nu}.
\end{equation*}
\end{Lemma}
\noindent The proof is a copy of the previous argument.

The Lemmas \ref{sum(lor)}, \ref{C1-C2}, \ref{C2-C3} and \ref{C3-C4} imply

\begin{Lemma}\label{comb} For every admissible pair $\bold\Pi_1,\,\bold\Pi_2$,
\[
P(\bold\Pi_1,\bold\Pi_2)-P_{\mathcal C_4}(\bold\Pi_1,\bold\Pi_2)\le_b e^{-\Theta(\log^2n)}\,2^{\mu} \left(\frac{e}{n+m}\right)^{\frac{n+m}{2}}\!\!n^{-\nu}.
\]
Here  $P_{\mathcal C_4}(\bold\Pi_1,\bold\Pi_2)$ is the integral of $F(\bold x,\bold y)$ over 
\[
\mathcal C_4\subset \{\bold x,\,\bold y\in [0,1]^n:\, x_i=y_i\ \text{ if }\,\bold\Pi_1(i)=\bold\Pi_2(i)\},
\]
defined by the additional constraints:  denoting $\xi:=\sum_{i\in \mathcal M} x_i'\,\left(=\sum_{i\in \mathcal M}x_i\right)$,
\begin{align}
s&:=\sum_{i\in [n]} x_i'+\sum_{j\in A\cup B} y_j' (=\xi_1+\xi_2+\xi_3)\le s_n (=n^{1/2}+6\log n), \label{1}\\
&\max\left\{\max_{i\in [n]} x_i',\,\max_{j\in A\cup B} y_j'\right\}\le 1.01\frac{s\log^2 n}{n},
\label{2}\\
&\left|\frac{|\mathcal M|}{2\xi^2}\sum\limits_{i\in\mathcal M} x_i^2-1\right|\le n^{-\sigma},\quad
\left|\frac{2|\mathcal M|}{\xi^2}\sum\limits_{(i,j)\in M_1\cap M_2}\!\!\!\!\!\!\!\!x_ix_j-1\right|\le n^{-\sigma}.
\label{3}
\end{align}
\end{Lemma}

The constraint \eqref{3} involves only $\{x_i\}_{i\in \mathcal M}$, and the constraint \eqref{2} imposes the bound for the individual components $x_i'$ and $y_j'$. Since $s_n\le 2n^{1/2}$, the latter implies that 
\begin{equation}\label{max[n],A,B<}
\max\left\{\max_{i\in [n]} x_i',\,\max_{j\in A\cup B} y_j'\right\}\le 3 n^{-1/2}\log^2 n,
\end{equation}
obviating the constraint $x_i'\le 1,\,y_j'\le 1$. On $\mathcal C_4$ the inequality 
\eqref{G(x,y)<expl} can be drastically simplified. First of all, the bottom part of the bound \eqref{G(x,y)<expl} is
\[
-\frac{1}{4}\Bigl(\sum_{i\in\mathcal M} x_i^2\Bigr)^2 +O\bigl(n^{-1/2}\log^5n\bigr).
\]
Second, 
\begin{align*}
\sum_{i\in\mathcal M}x_i^2&=\bigl(1+O(n^{-\sigma})\bigr)\frac{2\xi^2}{|\mathcal M|}=\frac{2\xi^2}{|\mathcal M|}+O(n^{-\sigma}),\\
\sum_{(i,j)\in M_1\cap M_2}\!\!\!\!\!\!\!\!x_ix_j&=\bigl(1+O(n^{-\sigma})\bigr)\frac{\xi^2}{2|\mathcal M|}
=\frac{\xi^2}{2|\mathcal M|} +O(n^{-\sigma}),
\end{align*}
{\it and\/} $\xi=\xi_1\bigl(1+O(n^{-1}\log^2 n)\bigr)$. In addition, $|\mathcal M|=n\bigl(1+O(n^{-1/2}\log n)\bigr)$. Therefore \eqref{G(x,y)<expl} becomes
\begin{align*}
G(\bold x,\bold y)&\le \bigl(1+O(n^{-\sigma})\bigr)\exp\bigl[H(\boldsymbol\xi)\bigr],\,\,
H(\boldsymbol\xi)=-\frac{s^2}{2}+\xi_2\xi_3+\frac{3\xi_1^2}{2n}
-\frac{\xi_1^4}{n^2}.
\end{align*}
%The power of Lemma \ref{comb} derives from 
\begin{Lemma}\label{P*sim}
\begin{align*}
&\pr_{\mathcal C_4}(\bold\Pi_1,\bold\Pi_2)\le \frac{2^{\mu}\bigl(1+O(n^{-\sigma})\bigr)}
{(n+m-1)!\,\bigl[(\nu-1)!\bigr]^2}\cdot \mathcal I(n+m,\nu),\\
&\mathcal I(n+m,\nu):=\iiint\limits_{(\xi_1,\xi_2,\xi_3)\in \bold R}\!
\exp\bigl[H(\boldsymbol{\xi})\bigr]\,\xi_1^{n+m-1}\cdot(\xi_2\xi_3)^{\nu-1}\,d\boldsymbol{\xi},\\
&\,\,\bold R:=\bigl\{\boldsymbol\xi\ge\bold 0:\, \xi_1\le n^{1/2}+6\log n;\,\,\xi_2, \xi_3\le 2\log^3 n\bigr\}.
\end{align*}
\end{Lemma}
\noindent The proof, of course, is based on the description of $\mathcal C_4$, and it runs along the familiar lines
of our preceding proofs; in particular, see the proof of Lemma \ref{C2-C3}. 
%Leaving the
%range $R$, and the additional terms in the exponent,  aside, the integral already appeared in \eqref{sumxi_j}. 
We omit the details. Furthermore, by the asymptotic formula for $\mathcal I(n,\nu)$ from \cite{Pit2} (3.60), we have 
\begin{align*}
\mathcal I(n+m,\nu)=(1+o(1))\left(\frac{\pi e}{n+m}\right)^{1/2}\!\left(\frac{n+m}{e}\right)^{\frac{n+m}{2}}
%&\quad\times 
\!\!\!(n+m)^{-\nu}\bigl[(\nu-1)!\bigr]^2.
\end{align*}
So, by Lemma \ref{P*sim},
\begin{equation*}%\label{P*<expl}
\begin{aligned}
\pr_{\mathcal C_4}(\bold\Pi_1,\bold\Pi_2)&\le  \frac{2^{\mu}\,\bigl(1+O(n^{-\sigma})\bigr)}{(n+m-1)!}
%&\,\,\times
\left(\frac{\pi e}{n+m}\right)^{1/2}\!\left(\frac{n+m}{e}\right)^{\frac{n+m}{2}}\!\!(n+m)^{-\nu}.
\end{aligned}
\end{equation*}
Therefore, within the factor $1+O(n^{-\sigma})$,
\begin{equation*}%\label{sumP*<expl}
\begin{aligned}
&\sum_{\bold\Pi_1,\bold\Pi_2}P_{\mathcal C_4}(\bold\Pi_1,\bold\Pi_2)\\
&\le \sum_{m\le m_n} \binom{n}{m}\,
\frac{f(m)}{(n+m-1)!}
\left(\frac{\pi e}{n+m}\right)^{1/2}\left(\frac{n+m}{e}\right)^{\frac{n+m}{2}}\\
&\,\,\,\,\,\times\sum_{\nu\le m_n}\binom{n-m}{2\nu}(n-m-2\nu-1)!!\,(n+m)^{-\nu}
\cdot\sum_{\mu}2^{2\mu}f(2\nu,\mu)
\end{aligned}
\end{equation*}
\begin{equation*}
\begin{aligned}
&=n!\sum_{m\le m_n}\frac{f(m)}{m!\,(n+m-1)!}\left(\frac{\pi e}{n+m}\right)^{1/2}\left(\frac{n+m}{e}\right)^{\frac{n+m}{2}}\\
&\,\,\,\,\,\times\sum_{\nu\le m_n}\frac{(n-m-2\nu-1)!!}{(n-m-2\nu)!\,(2\nu)!}
(n+m)^{-\nu}
\cdot\sum_{\mu}2^{2\mu}f(2\nu,\mu).
\end{aligned}
\end{equation*}
cf.  \eqref{clumsy}. The sum over $\mu$ is $e^{-1}(2\nu)! (1+O(1/\nu))$, see \eqref{sum_{mu}}. So the
sum over $\nu$ is asymptotic to
\begin{align*}
&\frac{e^{-1}}{(n-m)!!}\sum_{\nu\le m_n}(n+m)^{-\nu}\prod_{j=0}^{\nu-1}(n-m-2j)
\sim
\frac{e^{-1}}{(n-m)!!}\sum_{\nu\le m_n}e^{-\frac{\nu^2}{n}-\frac{2\nu m}{n}}.
%\\
%&=\frac{e^{-1} e^{m^2/n}}{(n-m)!!}\sum_{\nu=\nu_1}^{\nu_2}\exp\left(-\frac{(\nu+m)^2}{n}\right).
%\\
%&\sim e^{-1}\frac{\sqrt{\pi n}}{2} \frac{e^{m^2/n}}{(n-m)!!}.
%&e^{-1} (n-m-1)!!\sum_{\nu=\nu_1}^{\nu_2}(n+m)^{-\nu}\prod_{j=0}^{\nu-1}(n-m-2j)\\
%&=e^{-1} (n-m-1)!!\sum_{\nu=\nu_1}^{\nu_2}\exp\bigl(-\nu^2/n+O(\nu_2^3/n^2)\bigr)\\
%&=(1+o(1))(n-m-1)!!\,\frac{\sqrt{\pi n}}{2}.
\end{align*}
Thus, since $f(m)/m!=e^{-1}\sqrt{\frac{2}{\pi m}}(1+O(m^{-1}))$, the $m$-term in the resulting sum %over $m\in [m_1,m_2]$ 
is (within a factor $1+O(m^{-1})$)
\begin{align*}
&e^{-2}\sqrt{\frac{2}{\pi m}}\cdot\frac{n!}{(n-m)!!\,(n+m-1)!}
\left(\frac{ e}{n+m}\right)^{1/2}\left(\frac{n+m}{e}\right)^{\frac{n+m}{2}}\\
&\times \sum_{\nu\le m_n}e^{-\frac{\nu^2}{n}-\frac{2\nu m}{n}}\\
&
\sim e^{-3/2}\sqrt{\frac{2}{\pi^2 m}}\cdot e^{-\frac{m^2}{2n}}\sum_{\nu\le m_n}e^{-\frac{\nu^2}{n}-\frac{2\nu m}{n}}.
\end{align*}
So
\begin{equation*}
\begin{aligned}
&\sum_{\bold\Pi_1,\bold\Pi_2}\!\!P_{\mathcal C_4}(\bold\Pi_1,\bold\Pi_2)
\lesssim e^{-3/2}\sqrt{\frac{2}{\pi^2}}\sum_{m,\,\nu\le m_n}\!\!\!\!\frac{1+O(m^{-1})}{m^{1/2}}e^{-\frac{m^2}{2n}-\frac{2\nu m}{n}-\frac{\nu^2}{n}} \sim c n^{3/4},\\
 &\qquad\qquad\quad c:=e^{-3/2}\sqrt{\frac{2}{\pi^2}}\iint\limits_{x,\,y\ge 0} x^{-1/2}e^{-\frac{x^2}{2}-2xy -y^2}\,dxdy\approx
 0.617.
\end{aligned}
\end{equation*}
Thus, since $E[\mathcal S_n]$ is of order $n^{1/4}$, we have
\begin{Theorem}\label{E[Sn^2]sim} $\ex\bigl[\mathcal S_n^2\bigr]\lesssim cn^{3/4}$.
\end{Theorem}
\noindent With extra work, we could have proved that  $\ex\bigl[\mathcal S_n^2\bigr]\gtrsim cn^{3/4}$, as
well. Since $\ex\bigl[\mathcal S_n^2\bigr]\gg \ex^2[\mathcal S_n]$, we cannot deduce that $\mathcal S_n\to\infty$ in probability, even though $\ex[\mathcal S_n]\to\infty$. We firmly believe that the argument itself may help to define  a subset  of stable partitions for which the two-moments approach
will work just fine. For now we are content to use the techniques above to prove a result that would have been out of reach if not for the analysis of $\ex\bigl[\mathcal S_n^2]$.  
\begin{Theorem}\label{mult} Let $q_n$ denote the fraction of members that have more than one
stable predecessor. Then $\ex[q_n]\lesssim 2ec\, n^{-1/4}$, so that with high probability almost all members have a unique stable predecessor.
\end{Theorem}
\begin{proof} It suffices to consider the members outside the odd cycles. If any such member has some two stable partners, it belongs to a cycle of even length $\ge 4$ formed by the alternating
pairs matched in the corresponding stable partitions  $\bold \Pi_1$ and $\bold\Pi_2$. Notice that selecting every other edge of those cycles, we get a stable partition. Therefore, without loss of generality
we can  assume that $\bold\Pi_1$ and $\bold\Pi_2$ form a unique cycle of even length $2\nu\ge 4$. It follows that $Q_n$, the total number of members with at least two stable partners, is below the total
length of the single cycles formed by these special pairs of stable partitions $\bold\Pi_1$ and $\bold\Pi_2$. The bound does look crude, but it works. 

To bound the total expected length of those cycles, we need to estimate $\sum
2\nu(\bold\Pi_1,\bold\Pi_2) P_{\mathcal C_4}(\bold\Pi_1,\bold\Pi_2)$.
For those pairs we have $\mu:=\mu(\bold\Pi_1,\bold\Pi_2)=1$, and $\sum_{\mu}2^{2\mu}f(2\nu,\mu)=2(2\nu-1)!$.
Therefore
\[
\ex[Q_n]\lesssim \sum_{\bold\Pi_1,\bold\Pi_2}2\nu(\bold\Pi_1,\bold\Pi_2) P_{\mathcal C_4}(\bold\Pi_1,\bold\Pi_2)\lesssim 2ec\, n^{3/4}.
\]
%i.e. $\ex[q_n]\lesssim 2ec\, n^{-1/4}$.
\end{proof}
{\bf Acknowledgment.\/}  Almost thirty years ago Don Knuth introduced me to his ground-breaking
work on random stable marriages. Don's ideas and techniques have been a source of inspiration
for me ever since. The masterful book by Dan Gusfield and Rob Irving encouraged me to continue working on stable matchings. I am very grateful to Rob for the chance to work with him on
the stable roommates problem back in $1994$, and for his encouragement these last
months. Itai Ashlagi, Peter Bir\'o, Jennifer Chayes, Gil Kalai, Yash Kanoria, Jacob Leshno and the 
recent monograph by David Manlove made me aware of a significant progress
 in theory and applications of two/one--sided stable matchings. I thank the organizers for a valuable
 opportunity to participate in MATCH-UP 2017 Conference.


\begin{thebibliography}{99}
\bibitem{AbrBirMan}
D. J. Abraham, P. Bir\'o and D. F. Manlove, \textit{``Almost stable'' matchings in the roommates problem},
Proceedings of WAOA '05: the 3rd Workshop on Approximation and Online Algorithms, Lecture Notes
in Computer Science, \textbf{3879} (2006) 1--14. 
\bibitem{Alc}
J. Alcalde, \textit{Exchange-proofness or divorce proofness? Stability in one-sided matching markets}, Economic Design \textbf{1} (1995) 275--287.
\bibitem{AndAskRoy}
G. E. Andrews, R. Askey and R. Roy, \textit{Special Functions}, Cambridge University Press (1999).
\bibitem{CecMan}
K. Cechl\'arov\'a and D. F. Manlove, \textit{The exchange-stable marriage problem}, Disc. Appl. Math.
\textbf{152} (2005) 109--122.
\bibitem{FlaSed}
P. Flajolet and R. Sedgewick, \textit{Analytic Combinatorics}, Cambridge University Press (2009).
\bibitem{GusIrv}
D. Gusfield and R. W. Irving, \text{The Stable Marriage Problem, Structure and Algorithms}, The MIT
Press (1989).
\bibitem{IrvPit}
R. W. Irving and B. Pittel, \textit{An upper bound for the solvability probability of a random stable roommates instance}, \textbf{5} (1994) 465--486.
\bibitem{Knu}
D. E. Knuth, \textit{Stable marriage and its relation to other combinatorial problems: an introduction
to the mathematical analysis of algorithms}, CRM Proceedings and Lecture notes (1996).
\bibitem{KnuMotPit}
D. E. Knuth, R. Motwani and B. Pittel, \textit{Stable husbands}, Random Struct Algorithms \textbf{1}
(1991) 1--14.
\bibitem{Man}
D. F. Manlove, \textit{Algorithmics of Matching under Preferences}, World Scientific (2013).
\bibitem{Mer}
S. Mertens, \textit{Random stable matchings}, J. Statist. Mechanics: Theory and Experiment \textbf{10}
(2005).
\bibitem{Pit0}
B. Pittel, \textit{The average number of stable matchings}, SIAM J Disc Math \textbf{2} (1989)
530--549.
\bibitem{Pit1}
B. Pittel, \textit{On a random instance of a ``stable roommates'' problem: Likely behavior of the proposal
algorithm}, Comb Probab Comput \textbf{2} (1993) 53--92.
\bibitem{Pit2}
B. Pittel, \textit{The ``stable roommates'' problem with random preferences},  Ann Probab  \textbf{21}  (1993) 1441--1477.
\bibitem{Tan}
J. J. M. Tan, \textit{A necessary and sufficient condition for the existence of a complete stable matching}, J. Algorithms \textbf{12} (1991) 1--25.
\bibitem{Tan1}
J. J. M. Tan, \textit{Stable matchings and stable partitions}, International J. Computer Math.
\textbf{39} (1991) 11--20.
\end{thebibliography}
\end{document}